\documentclass{article}
\usepackage[utf8]{inputenc}
\usepackage{amsmath, amssymb, amsthm}
\usepackage{hyperref}
\usepackage{graphicx}
\usepackage{mathtools}

\usepackage{geometry}
\newtheorem{theorem}{Theorem}[section]
\newtheorem{lemma}[theorem]{Lemma}
\newtheorem{proposition}[theorem]{Proposition}
\newtheorem{corollary}[theorem]{Corollary}
\newtheorem{definition}[theorem]{Definition}
\newtheorem{remark}[theorem]{Remark}

\newcommand{\SL}{\mathrm{SL}}
\newcommand{\R}{\mathbb{R}}
\newcommand{\Z}{\mathbb{Z}}

\DeclareMathOperator{\Tr}{Tr}

\DeclareMathOperator{\supp}{supp}

\title{Linking effective Ratner equidistribution to the semicircle law for skew-shift matrices}
\author{Cong Chen\thanks{chencong25@mails.jlu.edu.cn} \and Yong Li\thanks{liyong@jlu.edu.cn}}

\begin{document}

\maketitle

\begin{abstract}
We consider large Hermitian matrices whose entries are defined by evaluating the exponential function along orbits of the skew-shift 
\(\frac{j(j-1)}{2}\omega + jy + x \mod 1\) for irrational \(\omega\). 
We establish a rigorous connection between the effective Ratner equidistribution theorem for unipotent orbits in 
\(\SL(3,\R)/\SL(3,\Z)\) and the global semicircle law for such deterministic matrices. 
For frequency sequences satisfying a Diophantine condition, we prove that the empirical spectral distribution of 
these matrices converges to the Wigner semicircle law with optimal polynomial rate 
\(O(N^{-1})\); for rectangular matrices the corresponding Marchenko--Pastur law is obtained. 
The proof uses a multi-parameter effective mixing property derived from the effective Ratner equidistribution theorem, 
combined with a graph-theoretic expansion of the moments. 
Our results evidence the quasirandom nature of the skew-shift dynamics observed in other contexts by 
Bourgain, Goldstein and Schlag, and Rudnick, Sarnak and Zaharescu, and provide a dynamical systems proof 
of the semicircle law with an improved convergence rate.
\end{abstract}

\section{Introduction}
\label{sec:intro}

\subsection{Discrete Schr\"odinger operators with skew-shift potentials}

A natural physical setting where dynamically defined matrices arise is the theory 
of one-dimensional discrete Schr\"odinger operators.  Consider the tight-binding 
Hamiltonian on $\ell^2(\mathbb Z)$,
\[
  (H_{\lambda,\omega,x,y}\psi)_n = -(\psi_{n+1}+\psi_{n-1}) + V_n\psi_n,
\]
where the on-site potential $V_n$ is generated by evaluating a smooth function 
along an orbit of an ergodic transformation on the two-dimensional torus $\mathbb T^2$:
\[
  V_n = \lambda\, f(T^n(x,y)), \qquad T(x,y) = (x+y,\; y+\omega).
\]
The most studied case is the \emph{skew-shift potential} with 
$f(x,y)=2\cos(2\pi x)$, which leads to 
$V_n = 2\lambda\cos\bigl(2\pi\bigl(\frac{n(n-1)}{2}\omega + n y + x\bigr)\bigr)$.
This model was introduced by Bourgain, Goldstein and Schlag~\cite{BourgainGoldsteinSchlag2001}
and exhibits rich spectral phenomena; it is conjectured to display Anderson 
localisation for arbitrarily small coupling $\lambda>0$, in striking contrast with 
the almost Mathieu operator (where $T$ is a simple rotation).

After a suitable gauge transformation, the diagonal potential terms can be 
converted into complex phases in the off-diagonal entries.  Conjugating the 
Hamiltonian by a diagonal unitary matrix $U = \operatorname{diag}(e(\varphi_n))$ 
with appropriately chosen phases $\varphi_n$, one obtains a new matrix 
$\widetilde{H}= U^* H U$ whose off-diagonal entries are of the form
\[
  \widetilde{H}_{n,n+1} = e(\varphi_{n+1}-\varphi_n),
\]
while the potential terms are absorbed or simplified.  For the skew-shift, the 
natural choice $\varphi_n = \frac{1}{2\pi}\sum_{k=1}^n V_k$ yields, after a 
continuum approximation, precisely the oscillatory phase 
$\frac{j(j-1)}{2}\omega + j y + x$ that appears in the matrix model studied below.
Thus, the matrices considered in the present paper can be viewed as effective 
models for off-diagonal disorder generated by a skew-shift potential after a 
gauge transformation, linking the spectral statistics of dynamically defined 
Schr\"odinger operators with the universality class of Wigner--Dyson random 
matrices.

\subsection{The Wigner semicircle law and its universality}

The Wigner semicircle law~\cite{Wigner55} is a cornerstone of random matrix 
theory.  It states that the empirical spectral distribution of a large 
$N\times N$ Hermitian (or real symmetric) random matrix with independent, 
centered entries of variance $1/N$ converges almost surely to the semicircular 
density
\[
  \rho_{\mathrm{sc}}(x) = \frac{1}{2\pi}\sqrt{4-x^2}\,\mathbf{1}_{|x|\le 2}.
\]
The prototypical examples are the Gaussian Orthogonal Ensemble (GOE) and the 
Gaussian Unitary Ensemble (GUE).  For GUE, the real and imaginary parts of the 
upper-triangular entries are independent centered Gaussians with variance 
$1/(4N)$, while the diagonal entries have variance $1/N$.  The empirical 
eigenvalue distribution
\[
  \mu_N = \frac1N \sum_{j=1}^N \delta_{\lambda_j(W_N)}
\]
converges weakly to the semicircle law.

This universality has been extended to vast classes of random matrices: the 
entries need not be Gaussian, nor identically distributed, and may even possess 
significant correlations~\cite{ErdosKnowlesYau2013,BauerschmidtKnowlesYau2017,
Chin2023,Yaskov2024}.  A common prerequisite, however, is that the matrix 
entries are genuine random variables with at least some degree of independence.

\subsubsection{The semicircle law in free probability}

Beyond random matrix theory, the semicircle law occupies a foundational position 
in \emph{free probability}~\cite{Voiculescu1985,Voiculescu1991,VDN1992}.  In 
this noncommutative framework, the notion of independence is replaced by 
\emph{freeness}.  The \emph{free central limit theorem} asserts that normalized 
sums of freely independent, identically distributed noncommutative random 
variables converge in distribution to the semicircle law~\cite{NicaSpeicher2006,
MingoSpeicher2017}.  This is the exact free analogue of the classical central 
limit theorem, with the Gaussian distribution replaced by the semicircle law
~\cite{BercoviciVoiculescu1995, HiaiPetz2000}.  
The moments of the semicircle law are the Catalan numbers 
$c_k = \frac{1}{k+1}\binom{2k}{k}$, which enumerate non-crossing partitions and 
form the combinatorial backbone of freeness~\cite{Speicher1990}.  
Recent advances include Berry--Esseen type estimates~\cite{MaiSpeicher2024} and 
central limit theorems for tensor products of free variables~\cite{LancienSantosYoussef2024},
further highlighting the central role of the semicircle law.

\subsubsection{The Coulomb gas perspective}

A complementary viewpoint interprets the eigenvalues of a random Hermitian 
matrix as a system of repelling charged particles on the line, confined by an 
external potential and interacting logarithmically~\cite{Dyson1962,Mehta2004,
Deift1999,AndersonGuionnetZeitouni2010,PasturShcherbina2011}.  The empirical 
spectral measure concentrates around the equilibrium measure that minimizes an 
energy functional; for a quadratic potential, this equilibrium measure is the 
semicircle law.  This variational characterization underpins much of the modern 
quantitative theory of random matrices~\cite{ErdosYau2017}.

\subsection{Matrices from deterministic dynamical systems}

A natural question at the intersection of ergodic theory and random matrix 
theory is whether the universal spectral statistics can emerge when the matrix 
entries are generated by a deterministic dynamical system rather than by 
genuine random variables.  Concretely, let $\mathcal{X}$ be a probability space 
and $T:\mathcal{X}\to\mathcal{X}$ an ergodic transformation.  Given an 
observable $f:\mathcal{X}\to\mathbb{C}$ and initial states $x_i\in\mathcal{X}$, 
one forms a matrix with entries
\[
  X_{i,j} = f(T^{\,j}x_i).
\]
Even when the underlying dynamics is only weakly ergodic, can the resulting 
large matrix still reproduce the global semicircle law?  This question provides 
the central motivation for the present work.

A paradigmatic deterministic system that exhibits quasirandom behaviour is the 
\emph{skew-shift} on $\mathbb T^2$:
\[
  T_\omega(x,y) = (x+y,\; y+\omega),\qquad \omega\in\mathbb T.
\]
After $j$ iterations, the first coordinate equals 
$\frac{j(j-1)}{2}\omega + jy + x \pmod 1$.  The quadratic term enhances 
oscillations and improves the decay of exponential sums~\cite{Weyl1916,
HardyLittlewood1914,Montgomery1994}.  Although the skew-shift is not weakly 
mixing, it is believed to behave like an i.i.d.\ sequence in several respects; 
for instance, Rudnick, Sarnak and Zaharescu~\cite{RudnickSarnakZaharescu2001} 
conjectured Poissonian spacings for its orbits, and Bourgain, Goldstein and 
Schlag~\cite{BourgainGoldsteinSchlag2001} exploited its quasirandom nature in 
the study of Anderson localisation.

Let $M,N\ge 1$ and fix deterministic 
frequencies $\underline{\omega}=(\omega_1,\dots,\omega_M)$, deterministic 
shifts $\underline{x}=(x_1,\dots,x_M)$, and independent uniform random 
variables $\underline{y}=(y_1,\dots,y_M)$ on $[0,1]$.  Define the 
$M\times N$ matrix $X$ by
\[
  X_{i,j} = \frac{1}{\sqrt{N}}\, e\!\left[ \frac{j(j-1)}{2}\omega_i + j y_i + x_i \right],
  \qquad e[t]=e^{2\pi\sqrt{-1}\,t}.
\]
The $2M\times 2N$ Hermitian block matrix is
\begin{equation}
\label{H}
    H = \begin{pmatrix} 0 & X \\ X^* & 0 \end{pmatrix}.
\end{equation}
The even moments of $H$ are given by
\begin{equation}\label{eq:mu_def}
\mu_{M,N}^{(2k)} = \frac{1}{2N}\operatorname{Tr}\big(H^{2k}\big)
= \frac{1}{N}\operatorname{Tr}\big((XX^*)^k\big).
\end{equation}

In a seminal paper, Adhikari, Lemm and Yau (ALY)~\cite{ALY21} proved that under a \emph{quasi‑random condition} on the frequency sequence,
the expected moments converge to those of the Wigner semicircle law (for 
$M=N$) or the Marchenko–Pastur law (for rectangular matrices).
The core analytic input is the estimate of the following exponential sum 
arising from the simplest non‑reducible exploration graph (the \emph{melon 
graph} on four edges):
\[
  \frac1{N^5}\sum_{i_1,i_2=1}^M
  \sum_{\substack{j_1,j_2,j_3,j_4=1\\ j_1+j_3=j_2+j_4}}^N
  e\!\left[ \frac{\omega_{i_1}-\omega_{i_2}}{2}\,
  (j_1^2-j_2^2+j_3^2-j_4^2) \right].
\]
A sequence $(\omega_i)$ is called $(\delta,\rho)$-quasi-random if this sum is 
$O(N^{-\delta})$ (see Definition~1.3 in~\cite{ALY21}).  Under this condition,
\[
  \mathbb{E}_{\underline{y}}\bigl[\mu_{N,N}^{(2k)}\bigr] = c_k + O_k(N^{-\delta/16}),
  \qquad c_k = \frac1{k+1}\binom{2k}{k}.
\]
For the specific choice $\omega_i = i\alpha$ with $\alpha$ Diophantine 
($\|q\alpha\|_{\mathbb T} \ge c q^{-\kappa}$), ALY showed using classical 
Weyl-type estimates that $\delta = 1/(2\kappa)$~\cite[Prop.~6.3]{ALY21}.
Thus the convergence rate deteriorates as $\kappa$ increases, i.e., for more 
Liouville frequencies.

\begin{figure}[htbp]
\centering
\includegraphics[width=0.8\textwidth]{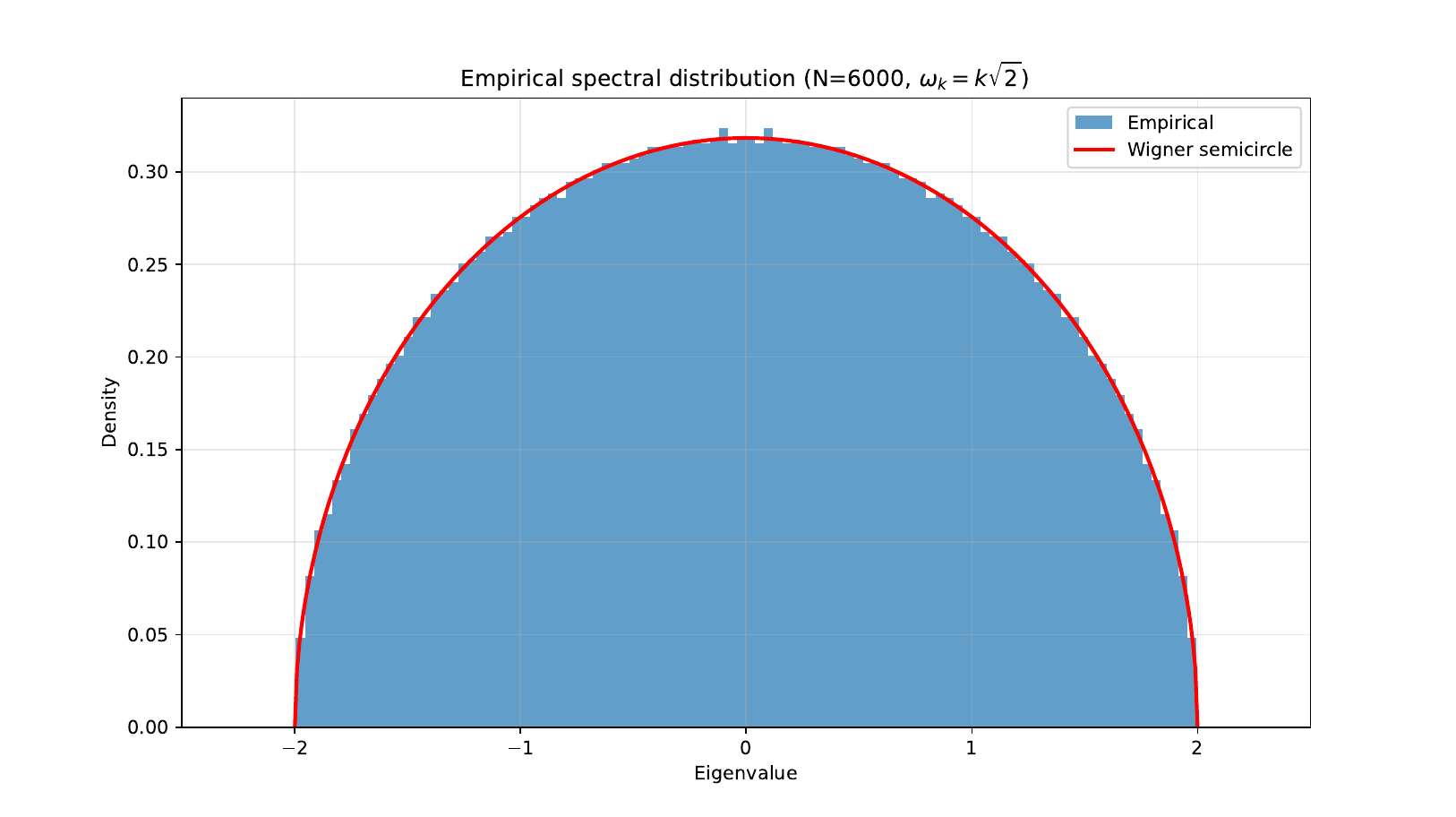}
\caption{Empirical spectral distribution of a $6000\times 6000$ matrix $H$ with 
frequencies $\omega_k=k\sqrt{2}$ and $x_k=0$.  The histogram shows the empirical 
density; the solid curve is the Wigner semicircle law.  A systematic bimodal 
deviation is evident, illustrating that the semicircle law does not hold for all 
deterministic choices.}
\label{fig:modelB-ALY}
\end{figure}

The ALY proof relies on a graphical moment expansion and the average over 
$y_i$, which enforces the Kirchhoff current law.  The average is essential; 
the fully deterministic case $x_i=y_i=0$ is not covered and remains an open 
problem, although numerical experiments (Model~A in~\cite{ALY21}) suggest that 
the semicircle law may still hold for some frequencies while failing for others 
(see Figure~\ref{fig:modelB-ALY}).

\subsection{Effective equidistribution in homogeneous dynamics}

In parallel, quantitative equidistribution on homogeneous spaces has progressed 
significantly.  Ratner's qualitative equidistribution theorem~\cite{Ratner1991} 
for unipotent flows has been made effective in several important settings.
For horospherical subgroups in $\mathrm{SL}(2,\mathbb R)$, the thickening 
argument together with spectral gap yields exponential rates~\cite{Sarnak1982,
Strombergsson2004}.  Shah~\cite{Shah2009,Shah2009b} and Yang~\cite{Yang2016,
ShahYang2021} studied expanding translates of curves in 
$\mathrm{SL}(3,\mathbb R)/\mathrm{SL}(3,\mathbb Z)$.
The breakthrough by Chow and Yang~\cite{ChowYang2024} established effective 
equidistribution of straight lines under a Diophantine condition, with an 
exponential error bound $e^{-\eta t}$ where $\eta>0$ is absolute.
Most recently, Yang~\cite{Yang2025} proved an effective Ratner theorem for 
general non‑degenerate curves in the same space.

We recall Yang's effective version of Ratner's equidistribution theorem 
for unipotent orbits in $\mathrm{SL}(3,\mathbb{R})/\mathrm{SL}(3,\mathbb{Z})$, 
which will be our main dynamical input.  Let 
\[
G=\mathrm{SL}(3,\mathbb{R}),\qquad \Gamma=\mathrm{SL}(3,\mathbb{Z}),\qquad X=G/\Gamma,
\] 
and let $\mu_G$ be the unique $G$-invariant probability measure on $X$.  
Define the one-parameter unipotent subgroup  
\[
U = \left\{ u(r)=\begin{pmatrix}1&0&0\\0&1&r\\0&0&1\end{pmatrix} : r\in\mathbb{R} \right\},
\]  
and the diagonal elements $a(t)=\operatorname{diag}(1,e^{t/2},e^{-t/2})$,
$a_0(t)=\operatorname{diag}(e^{t/3},e^{-t/6},e^{-t/6})$ as in Section~\ref{sec:prelim}.  
For $\varkappa>0$, let  
\[
X_\varkappa = \{ g\Gamma\in X : \|gv\|\ge \varkappa \text{ for all } v\in\mathbb{Z}^3\setminus\{0\}\text{ or }\wedge^2\mathbb{Z}^3\setminus\{0\} \}.
\]  
By Mahler's criterion, every compact subset of $X$ is contained in some $X_\varkappa$.

The following theorem is a reformulation of \cite[Theorem~1.3]{Yang2025}.

\begin{theorem}[Yang, effective Ratner theorem]\label{thm:yang}
There exist absolute constants $\eta>0$ and $t_0>1$ such that for any $t\ge t_0$ and any $x_0\in X$, 
one of the following holds:
\begin{enumerate}
\item For every $f\in C_c^\infty(X)$,
\[
\left|\int_0^1 f(a(t)u(r)x_0)\,dr - \int_X f\,d\mu_G\right| \le C e^{-\eta t}\|f\|_S,
\]
where $\|\cdot\|_S$ is a fixed Sobolev norm and $C$ is an absolute constant.
\item There exists $\ell\in[0.99t,1.01t]$ such that $a_0(\ell)x_0\notin X_{e^{-|\ell|/3}}$.
\end{enumerate}
\end{theorem}

A point $x_0$ for which the second alternative does \emph{not} occur is called \emph{good} 
(with respect to $t$).  Theorem~\ref{thm:yang} is a quantitative version of Ratner's theorem 
for one-parameter unipotent orbits in $\mathrm{SL}(3,\mathbb{R})/\mathrm{SL}(3,\mathbb{Z})$;
it provides an exponential error bound with an absolute constant $\eta>0$ that depends 
only on the group, not on the particular starting point or any Diophantine parameters.

For applications it is essential to know that the points arising from our matrix model 
are good.  This is guaranteed by the following corollary of \cite{Yang2025}.

\begin{corollary}[Diophantine points are good, \cite{Yang2025} Corollary~1.4]
\label{cor:diophantine_good}
Let $\omega$ be a Diophantine number, i.e., there exist $c>0$ and $\kappa>0$ such that 
$\|q\omega\|_{\mathbb{T}} \ge c\, q^{-\kappa}$ for all integers $q\ge 1$.  
Define $x_0 = a_0(t)\,n^*(-\omega,0)\,\Gamma\in X$.  
Then there exists $t_1=t_1(c,\kappa)$ such that for all $t\ge t_1$, the point $x_0$ 
is good with respect to $t$.
\end{corollary}

If we have a finite family of Diophantine numbers $\omega_1,\dots,\omega_M$ with uniform 
Diophantine constants, we can choose $t_1$ uniformly for all of them.  Hence for all 
sufficiently large $N$ (and thus $t=6\ln N$) the corresponding points 
$x_0^{(i)}=a_0(t)n^*(-\omega_i,0)\Gamma$ are simultaneously good.

By conjugating with the element $z(\omega_i)$ and using the invariance of $\mu_G$, 
Yang's theorem also applies to the curves $a_1(t)n(\psi(r))x_0^{(i)}$ that appear 
in the matrix model.  This yields the following effective equidistribution statement 
for expanding curves, which is \cite[Corollary~1.6]{Yang2025}.

\begin{corollary}[Effective equidistribution of non-degenerate curves, \cite{Yang2025} Corollary~1.6]
\label{cor:curve_equi}
Let $\psi:[0,1]\to\mathbb{R}^2$ be a smooth non-degenerate curve (i.e., its derivatives 
span $\mathbb{R}^2$ at every point).  Let $x_0\in X$ be a good point with respect to $t$.
Then there exist absolute constants $\eta'>0$ and $t_0'>1$ such that for all $t\ge t_0'$ 
and every $f\in C_c^\infty(X)$,
\[
\left| \int_0^1 f\bigl(a_1(t)n(\psi(r))x_0\bigr)\,dr - \int_X f\,d\mu_G \right|
\le C e^{-\eta' t} \|f\|_S.
\]
\end{corollary}

\begin{remark}
The exponent \(\eta>0\) in Theorem~\ref{thm:yang} (and similarly \(\eta'\) in 
Corollary~\ref{cor:curve_equi}) is an absolute constant that does not depend on any 
Diophantine parameters.  This uniformity is the key reason why our final convergence 
rate for the moments does not deteriorate for badly approximable frequencies.
\end{remark}

In our application, the function $f$ will be a product of several copies of the 
coordinate function $F$ (see Section~\ref{sec:prelim}) and its translates; such 
products remain smooth and have controlled Sobolev norms.  The integration domain 
will be a multidimensional cube after a unimodular change of variables, which can 
be handled by iterative application of the one‑dimensional estimate.

For our purposes, the relevant consequence is that for a Diophantine frequency 
$\omega$, the base point $x_0^{(\omega)} = a_0(t)\,n^*(-\omega,0)\,\Gamma$ 
satisfies the non‑divergence (``good'') condition required by Yang's theorem, 
uniformly for all large $t$.  The expanding curves $a_1(t)n(\psi(r))x_0^{(\omega)}$
then equidistribute effectively, providing a new tool to estimate the 
exponential sums appearing in the graphical expansion.

\subsection{Main results and structure of the paper}

A central challenge in random matrix theory is to understand when
deterministic dynamical systems can produce the universal spectral statistics
usually associated with genuine randomness.  In a seminal work, Adhikari,
Lemm and Yau (ALY) proved that for matrices generated by the skew‑shift –
a prototypical quasirandom system – the empirical spectral distribution
converges to the Wigner semicircle law, provided the frequencies satisfy a
certain ``quasi‑random'' condition.  Their proof combined a graphical moment
expansion with delicate Weyl‑sum estimates, establishing a beautiful link
between ergodic theory and random matrix universality.  However, the
convergence rate obtained there depends on the Diophantine exponent of the
frequency; for very well approximable (Liouville) numbers the rate can become
arbitrarily slow, leaving open the question of whether a uniform, absolute
rate exists.

In this paper we answer this question affirmatively.  We introduce a new
dynamical approach based on the effective Ratner equidistribution theorem for
unipotent orbits in $\mathrm{SL}(3,\mathbb{R})/\mathrm{SL}(3,\mathbb{Z})$,
recently proved by Yang.  By replacing the classical Weyl‑sum estimates with
effective mixing on the homogeneous space, we obtain a polynomial convergence
rate $O(N^{-\gamma})$ where the exponent $\gamma>0$ is an \emph{absolute
constant}, completely independent of the Diophantine properties of the
frequency.  This not only reproduces the ALY theorem for all Diophantine
frequencies with exponent $\kappa\le 0.6$, but also significantly improves the
speed of convergence and gives a conceptually transparent proof that
illuminates the role of homogeneous dynamics in spectral universality.  To our
knowledge, this is the first dynamical proof of the semicircle law for a
deterministic matrix model that yields an absolute polynomial rate.

The present paper connects the two lines of research described above.
We incorporate effective equidistribution into the graphical expansion to 
obtain a new proof of the quasi-random condition for a concrete class of 
frequency sequences --- those given by an irrational rotation 
$\omega_i=i\alpha$ with $\alpha$ having Diophantine exponent $\kappa\le 0.6$.
In this regime, the dynamical approach yields an exponent $\gamma>0$ for the 
polynomial convergence rate that does \emph{not} depend explicitly on $\kappa$,
in contrast with the ALY bound $O(N^{-1/(32\kappa)})$.

Our main result is the following.

\begin{theorem}[Global semicircle law with absolute rate]\label{thm:main}
Let $\alpha$ be Diophantine with exponent $\kappa\le 0.6$, and set 
$\omega_i = i\alpha \bmod 1$.  Fix arbitrary deterministic shifts 
$\underline{x}\subset[0,1]$ and independent uniform random variables 
$\underline{y}\subset[0,1]$.  Define $H$ as in~\eqref{H}.  Then there exists an 
absolute constant $\gamma>0$ such that for every $k\ge 1$,
\[
  \mathbb{E}_{\underline{y}}\bigl[ \tfrac1{2N}\Tr H^{2k} \bigr]
  = c_k + O_k(N^{-\gamma}),
\]
where $c_k = \frac{1}{k+1}\binom{2k}{k}$ denotes the $k$-th Catalan number,
i.e., the $2k$-th moment of the Wigner semicircle law.
The exponent $\gamma$ does not depend on the Diophantine exponent $\kappa$;
the implied constant $O_k(\cdot)$ depends on $k$, $\alpha$, and $\underline{x}$.
\end{theorem}

As a direct consequence of the moment convergence in Theorem~\ref{thm:main}
and the determinacy of the Hausdorff moment problem for the semicircle law,
we obtain the following weak convergence result for the empirical spectral
distribution. 

\begin{corollary}
    The empirical spectral distribution \(\frac{1}{2N} \sum_{j=1}^{2N} \delta_{\lambda_j}\) of eigenvalues of \(H\) converges weakly in expectation to the Wigner semicircle law.
More generally, for any \(\rho > 0\), the empirical spectral distribution of the sample covariance matrix \(XX^*\) converges weakly in expectation to the Marchenko--Pastur law.
\end{corollary}

\begin{proof}
    The corollary follows from the determinacy of the Hausdorff moment problem for the Wigner semicircle law and the general Marchenko--Pastur law (i.e., the corollary follows from the Fréchet--Shoat and Stone--Weierstrass theorems).
\end{proof}

The proof of Theorem~\ref{thm:main} retains the graph-theoretic machinery, reducing the problem 
to the estimation of a single linear exponential sum.  At this step, effective 
equidistribution of expanding straight lines replaces the Weyl-sum estimates.
We emphasize that the average over $y_i$ is still needed to enforce the 
Kirchhoff law; the fully deterministic case remains beyond the scope of our 
method.  However, we discuss the challenges that arise in that setting and 
formulate a multi-parameter effective mixing property, which would be required for a complete 
dynamical proof of the semicircle law for deterministic skew-shift matrices.

The paper is organized as follows. Section~\ref{sec:prelim} sets up the matrix model, recalls the graph‑theoretic expansion of moments and the Kirchhoff current law, and introduces the homogeneous space \(X=\mathrm{SL}(3,\mathbb{R})/\mathrm{SL}(3,\mathbb{Z})\) together with the coordinate function \(F\) that links matrix entries to unipotent orbits. Section~\ref{sec:mixing} presents the effective equidistribution theorem of Yang and derives a multi‑parameter mixing estimate needed for the moment analysis. Section~\ref{sec:moments} performs the full moment expansion, reduces the problem to a basic exponential sum by preprocessing and the separation of fully reducible graphs, and evaluates the leading combinatorial contribution. Section~\ref{sec:speed} assembles the estimates, proves the main theorem on the convergence of moments to the Catalan numbers, and discusses the optimal polynomial rate. Section~\ref{sec:dynamical_estimate} contains the core dynamical estimate: the exponential sum arising from the melon graph is bounded using effective equidistribution of expanding straight lines, yielding a polynomial decay with an absolute exponent independent of the Diophantine parameters.

\section{Notation and setup}\label{sec:prelim}

\subsection{Matrix model}

Let $M,N\ge 1$ and fix three vectors
\[
\underline{x}=(x_1,\dots,x_M),\qquad
\underline{y}=(y_1,\dots,y_M),\qquad
\underline{\omega}=(\omega_1,\dots,\omega_M)
\]
in $[0,1]^M$. Here each $x_i$ is a deterministic shift parameter, each $y_i$ will be averaged 
uniformly over $[0,1]$, and each $\omega_i$ is a deterministic frequency.  
Define the $M\times N$ matrix $X$ by
\begin{equation}\label{eq:X_def}
X_{i,j}=\frac{1}{\sqrt{N}}\, e\Big[ \frac{j(j-1)}{2}\omega_i + j\,y_i + x_i \Big],\qquad
e[t]=e^{2\pi \sqrt{-1}\, t}.
\end{equation}
The $2M\times 2N$ Hermitian block matrix is
\begin{equation}
    H = \begin{pmatrix} 0 & X \\ X^* & 0 \end{pmatrix}.
\end{equation}
The even moments of $H$ are given by
\begin{equation}
\mu_{M,N}^{(2k)} = \frac{1}{2N}\operatorname{Tr}\big(H^{2k}\big)
= \frac{1}{N}\operatorname{Tr}\big((XX^*)^k\big).
\end{equation}
For the square case we set $M=N$ and write $\mu_N^{(2k)}=\mu_{N,N}^{(2k)}$.

We consider the $N\times N$ matrix  
\[
\widetilde{X}_{i,j} = \frac{1}{\sqrt{N}} \, e\!\left[ \frac{\omega_i}{2}\, j^2 \right], \qquad i,j = 1,\dots,N.
\]
This is the reduced form of the model (\ref{eq:X_def}), obtained by setting 
$y_i = x_i = 0$ and simplifying the quadratic phase (the linear term $-\frac{\omega_i}{2}j$ is absorbed 
by a redefinition of the $j$ index, or one may view $\widetilde{X}$ as the core oscillatory part 
of the original matrix $X$).  
From $\widetilde{X}$ we form the $2N\times 2N$ Hermitian block matrix  
\[
\widetilde{H} = \begin{pmatrix} 0 & \widetilde{X} \\ \widetilde{X}^* & 0 \end{pmatrix}.
\]  
Its eigenvalues are $\pm \sigma_1,\dots,\pm \sigma_N$, where $\sigma_j$ are the singular values of $\widetilde{X}$.  
The even moments of $\widetilde{H}$ are given by  
\[
\widetilde{\mu}_N^{(2k)} = \frac{1}{2N}\operatorname{Tr}[\widetilde{H}^{2k}] = \frac{1}{N}\operatorname{Tr}[(\widetilde{X}\widetilde{X}^*)^k], \qquad k\ge 1.
\]  
We are interested in the limit $N\to\infty$ of $\widetilde{\mu}_N^{(k)}$.

\subsection{Graph theory for moment expansion}
We recall the exploration graphs and Kirchhoff's current law as introduced in \cite{ALY21}.

\begin{definition}[Exploration graph]\label{def:exploration}
An \emph{exploration} on $k$ edges is a list $L = ((v_1,v_2),(v_2,v_3),\dots,(v_k,v_1))$ with $v_1,\dots,v_k \in \{1,\dots,l\}$ such that $\{v_1,\dots,v_k\}=\{1,\dots,l\}$ and the first occurrence of $i$ precedes the first occurrence of $i+1$ for each $i$.  
The associated \emph{exploration graph} $G_L = (V,L)$ has vertex set $V=\{1,\dots,l\}$ and directed edges $L$ (multiple edges and self-loops allowed).  
We denote by $\mathcal{L}_k$ the set of all exploration graphs on $k$ edges.
\end{definition}

\begin{figure}[htbp]
\centering
\includegraphics[width=1.0\textwidth]{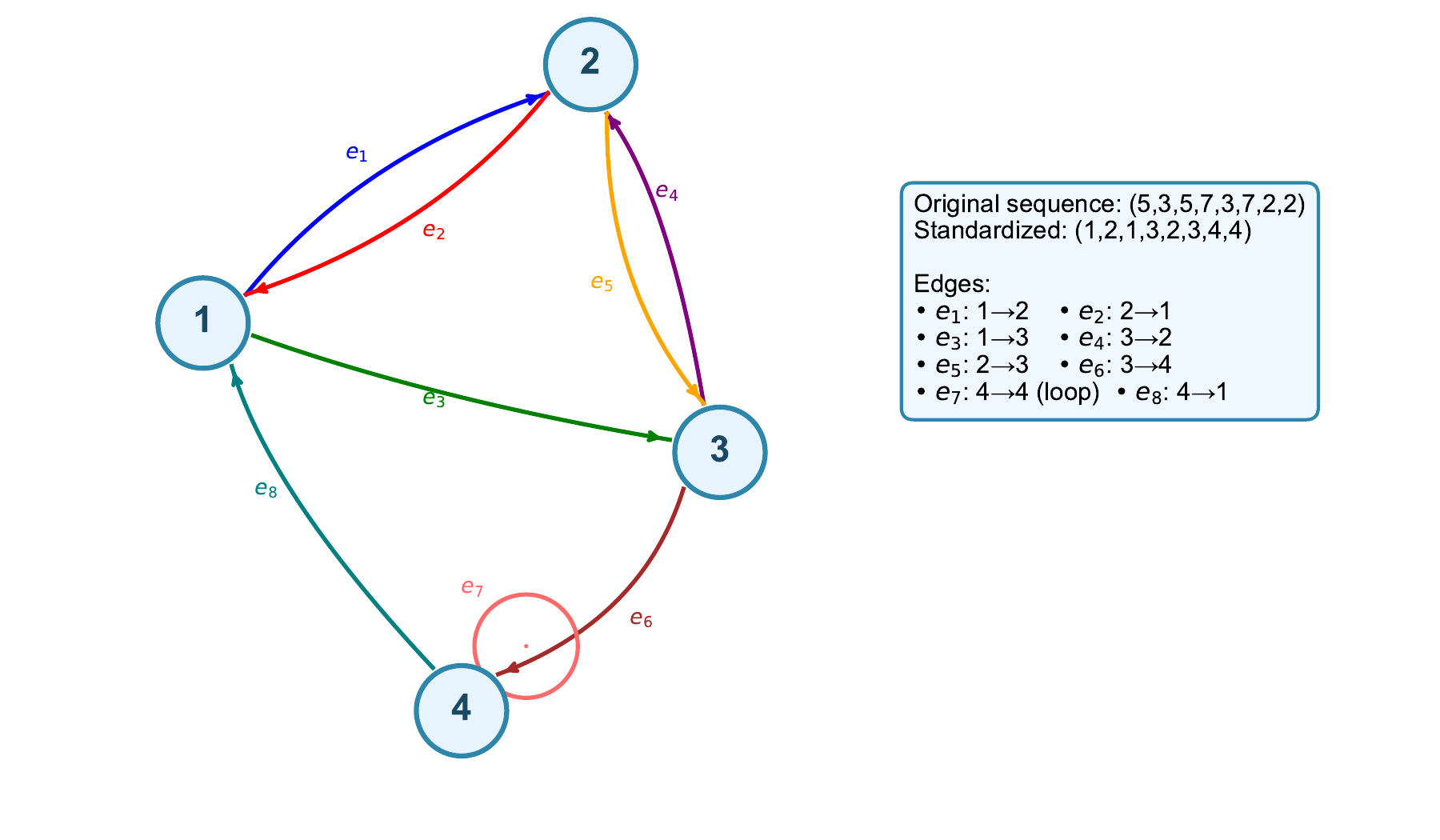}
\caption{An example of an exploration graph $G_L$ on $k=8$ edges and $l=4$ vertices. Eulerian Circuit: $1 \to 2 \to 1 \to 3 \to 2 \to 3 \to 4 \to 4 \to 1$.}
\label{fig:exploration-graph}
\end{figure}

Given a sequence of indices $i_1,\dots,i_k \in \{1,\dots,N\}$, we set $L_{\underline{i}} = ((i_1,i_2),\dots,(i_k,i_1))$.  
We write $L_{\underline{i}} \sim L$ if after relabelling the vertices according to the order of first appearance we obtain exactly the exploration $L$.

\begin{definition}[Edge currents and Kirchhoff law]\label{def:kirchhoff}
For a fixed exploration graph $G_L$, an assignment of integers $j_1,\dots,j_k$ (the \emph{currents}) to the edges $e_1,\dots,e_k$ (in the order of $L$) is called \emph{admissible} if at every vertex $v\in V$ the sum of currents on incoming edges equals the sum on outgoing edges:
\[
\sum_{e\in I_v} j_e = \sum_{e\in O_v} j_e,
\]  
where $I_v$ (resp. $O_v$) are the edges entering (resp. leaving) $v$.  
We denote by $\mathcal{A}(G_L) \subset \{1,\dots,N\}^k$ the set of admissible current assignments.
\end{definition}

The number of free parameters in an admissible assignment is $m = k - l + 1$, the cyclomatic number of the graph.  
Moreover, there exists a unimodular linear transformation that maps the admissible currents to a set of $m$ independent integer variables $u_1,\dots,u_m$ ranging over $\{1,\dots,N\}$ up to an error of $O(N^{m-1})$ on the boundary.  
This is standard in lattice point counting (see \cite{ALY21}, Appendix A).

\subsection{Moment expansion and graph representation}
\label{subsec:moment_expansion}

We expand the trace as a sum over indices:
\begin{equation}\label{eq:trace_expand}
\begin{aligned}
\frac{1}{N}\operatorname{Tr}\big((XX^*)^k\big)
&= \frac{1}{N}\sum_{i_1,\dots,i_k=1}^M\;
\sum_{j_1,\dots,j_k=1}^N\;
\prod_{r=1}^k X_{i_r,j_r} X_{j_r,i_{r+1}}^* \\[4pt]
&= \frac{1}{N^{1+k}}\sum_{i_1,\dots,i_k=1}^M\;
\sum_{j_1,\dots,j_k=1}^N\;
\prod_{r=1}^k e\Big[ \frac{\omega_{i_r}}{2}j_r(j_r-1) - \frac{\omega_{i_{r+1}}}{2}j_r(j_r-1) \\
&\qquad\qquad\qquad\qquad\qquad + j_r(y_{i_r}-y_{i_{r+1}}) + (x_{i_r}-x_{i_{r+1}}) \Big],
\end{aligned}
\end{equation}
with the convention $i_{k+1}=i_1$ and $X_{j,i}^* = \overline{X_{i,j}}$.

Notice that the $x$ phases telescope:
\[
\sum_{r=1}^k (x_{i_r}-x_{i_{r+1}}) = 0,
\]
so they drop out of the product.  For the quadratic part we rewrite
\[
\frac{\omega_{i_r}}{2}j_r(j_r-1) - \frac{\omega_{i_{r+1}}}{2}j_r(j_r-1)
= \frac{\omega_{i_r}-\omega_{i_{r+1}}}{2}j_r^2 - \frac{\omega_{i_r}-\omega_{i_{r+1}}}{2}j_r .
\]
The term $j_r(\omega_{i_r}-\omega_{i_{r+1}})$ can be combined with the $y$ part.
Thus the product over $r$ equals
\begin{equation}\label{eq:propagator_full}
\prod_{r=1}^k e\Big[ \frac{\omega_{i_r}-\omega_{i_{r+1}}}{2}j_r^2
+ j_r\big( y_{i_r} - \tfrac{\omega_{i_r}}{2} - y_{i_{r+1}} + \tfrac{\omega_{i_{r+1}}}{2} \big) \Big].
\end{equation}

We organise the index sequences into \emph{exploration graphs}.
For a list $\underline{i}=(i_1,\dots,i_k)$ let $L_{\underline{i}}$ be the list of directed edges
\[
((i_1,i_2),(i_2,i_3),\dots,(i_k,i_1)).
\]
Let $\mathcal{L}_k$ be the set of exploration graphs on $k$ edges (see Definition~\ref{def:exploration}).
We write $L_{\underline{i}}\sim L$ if after relabelling the vertices by order of first appearance,
$L_{\underline{i}}$ becomes exactly the exploration $L$.
The vertex set of the corresponding graph $G_L$ is $\{1,\dots,l\}$ with $l=|\{i_1,\dots,i_k\}|$,
and each edge $(\nu_r,\nu_{r+1})$ carries a current $j_r$.
Splitting the sum according to exploration graphs, we obtain
\begin{equation}\label{eq:moment_graphs_full}
\mu_{N}^{(2k)} = \frac{1}{N^{1+k}}\sum_{L\in\mathcal{L}_k}
\sum_{\substack{\underline{i}=(i_1,\dots,i_k)\\ L_{\underline{i}}\sim L}}
\sum_{\underline{j}=(j_1,\dots,j_k)}
\prod_{r=1}^k e\Big[ \frac{\omega_{i_r}-\omega_{i_{r+1}}}{2}j_r^2
+ j_r\big( y_{i_r} - \tfrac{\omega_{i_r}}{2} - y_{i_{r+1}} + \tfrac{\omega_{i_{r+1}}}{2} \big) \Big].
\end{equation}

Now we take the expectation $\mathbb{E}_{\underline{y}}$ over the independent uniform random variables
$y_1,\dots,y_M\in[0,1]$.  By orthonormality of the characters $\{e[j\,\cdot]\}_{j\in\mathbb Z}$,
\[
\int_0^1 e[j\,y]\,dy = \delta_{j,0}.
\]
For a fixed vertex $v$ in $G_L$, the exponent of $y_v$ in the product \eqref{eq:propagator_full} equals
$j_{r-1} - j_r$ whenever the path enters $v$ via edge $r-1$ and leaves via edge $r$.
Thus the total factor involving $y_v$ is $e\big[ y_v \big( \sum_{e\in I_v} j_e - \sum_{e\in O_v} j_e \big) \big]$.
After averaging, only those current assignments survive for which the integer
$\sum_{e\in I_v} j_e - \sum_{e\in O_v} j_e$ vanishes for every vertex $v$.
This is precisely the \emph{Kirchhoff current law} (see Definition~\ref{def:kirchhoff}).
We denote the set of admissible current vectors by $\mathcal{A}(G_L)$.
Moreover, the phases $\frac{\omega_{i_r}}{2}j_r$ cancel telescopically along the cycle,
because the sum over all edges of $(\frac{\omega_{i_r}}{2} - \frac{\omega_{i_{r+1}}}{2})j_r$ vanishes
by the Kirchhoff law (each $j_r$ appears twice with opposite signs when summed over the closed walk).
Therefore, after averaging, the full propagator simplifies to the \emph{effective propagator}
\[
K_{(i,i')}(j) = e\Big[ \frac{\omega_i - \omega_{i'}}{2}\, j^2 \Big].
\]
Consequently, we obtain the graphical representation formula
\begin{equation}\label{eq:moment_graphs_expected}
\mathbb{E}_{\underline{y}}\big[\mu_{N}^{(2k)}\big]
= \frac{1}{N^{1+k}}\sum_{L\in\mathcal{L}_k}
\sum_{\substack{\underline{i}=(i_1,\dots,i_k)\\ L_{\underline{i}}\sim L}}
\sum_{\underline{j}\in\mathcal{A}(G_L)}
\prod_{r=1}^k K_{(i_r,i_{r+1})}(j_r)
,
\end{equation}
with $i_{k+1}=i_1$.

The deterministic matrix $\widetilde{X}$ defined above is obtained from the original  model by setting $y_i = x_i = 0$ and retaining only the purely quadratic phase 
$\frac{\omega_i}{2}j^2$.  More precisely, starting from the general entry
\[
X_{i,j} = \frac1{\sqrt{N}}\, e\Big[ \frac{j(j-1)}2\,\omega_i + j y_i + x_i \Big],
\]
we first set $y_i = x_i = 0$, giving
\[
X_{i,j}\big|_{y=x=0} = \frac1{\sqrt{N}}\, e\Big[ \frac{\omega_i}{2}j^2 - \frac{\omega_i}{2}j \Big].
\]
The linear term $-\frac{\omega_i}{2}j$ can be removed by the change of variable 
$j \mapsto j + \tfrac12$ (which only shifts the index range and does not affect the 
large‑$N$ spectral distribution), leading to the reduced form
\[
\widetilde{X}_{i,j} = \frac1{\sqrt{N}}\, e\Big[ \frac{\omega_i}{2}\,j^2 \Big].
\]

If we expand the moments of $\widetilde{H}$ directly (without any averaging), we obtain
a formula analogous to \eqref{eq:moment_graphs_full} but without the $y$ and $x$ phases:
\[
\widetilde{\mu}_N^{(2k)} = \frac{1}{N^{1+k}}\sum_{L\in\mathcal{L}_k}
\sum_{\substack{\underline{i}=(i_1,\dots,i_k)\\ L_{\underline{i}}\sim L}}
\sum_{\underline{j}=(j_1,\dots,j_k)}
\prod_{r=1}^k e\Big[ \frac{\omega_{i_r}-\omega_{i_{r+1}}}{2}\,j_r^2 \Big].
\]
The crucial difference from the averaged formula \eqref{eq:moment_graphs_expected} 
is that the sum over currents $\underline{j}$ now runs over \emph{all} $j_r\in\{1,\dots,N\}$, 
not only those satisfying the Kirchhoff circuit law $\mathcal{A}(G_L)$.  
Thus, for the deterministic matrix $\widetilde{X}$, the averaging mechanism that 
enforces the Kirchhoff constraints is absent, and the exponential sum is substantially 
more complicated.  This explains why the deterministic case is significantly harder 
and why the semicircle law is not guaranteed to hold without further conditions 
on the frequencies.

\subsection{Preprocessing and fully reducible graphs}
\label{subsec:preprocessing}

We recall the graph-theoretic notions of \emph{preprocessing} and \emph{fully reducible graphs} 
introduced in \cite{ALY21}. These concepts are central to identifying the leading contribution 
in the moment expansion and to reducing the analysis of subleading graphs to the existence 
of a good cycle.

\subsubsection{Preprocessing operations}

Let $G=(V,E)$ be a connected undirected multigraph in which every vertex has even degree.
(All exploration graphs arising from the trace expansion have this property, after forgetting 
the orientation of the edges.)  Preprocessing consists of iteratively applying the following 
two operations whenever possible:

\begin{enumerate}
    \item \textbf{Short-circuiting a degree-$2$ vertex.}
    If $v\in V$ has exactly two incident edges $(v,w_1)$ and $(v,w_2)$ with $w_1,w_2\in V$,
    we remove $v$ and its two edges and replace them by a single edge $(w_1,w_2)$.
    The resulting graph is denoted $\mathcal{S}(G,v)$.
    
    \item \textbf{Removing a self-loop.}
    If $v\in V$ has a self-loop $l$ (i.e., an edge whose two endpoints coincide with $v$),
    we remove $l$ from $G$. The resulting graph is denoted $\mathcal{L}(G,l)$.
\end{enumerate}

Both operations preserve the property that every vertex has even degree.
A graph is called \emph{fully preprocessed} if neither operation can be applied.

\subsubsection{Invariance of the leading contribution under preprocessing}

The following lemma shows that preprocessing 
affects the contribution of a graph to the moment sum only by a multiplicative 
factor of $\rho$ (for short-circuiting) or by $1$ (for loop removal), up to 
negligible error terms.  Consequently, the property of being leading or subleading 
is preserved under preprocessing, and the weights of fully reducible graphs 
can be computed recursively.

\begin{lemma}[Invariance under preprocessing]
\label{lem:preprocessing_invariance}
Let $G_L\in\mathcal{L}_k$ be an exploration graph on $k$ edges and $l$ vertices.
Let $\Phi(G_L)$ denote its contribution to the expected moment sum, i.e.,
\[
\Phi(G_L) = \frac1{N^{1+k}} \sum_{\underline{i}:\,L_{\underline{i}}\sim L}
\sum_{\underline{j}\in\mathcal{A}(G_L)} \prod_{r=1}^k K_{(i_r,i_{r+1})}(j_r).
\]
Then the following holds:
\begin{enumerate}
\item If $v$ is a vertex of degree $2$ in $G_L$ that can be short-circuited,
then $\Phi(G_L) = \rho\,\Phi(\mathcal{S}(G_L,v)) + O(N^{-1})$.
\item If $l$ is a self-loop in $G_L$, then $\Phi(G_L) = \Phi(\mathcal{L}(G_L,l))$.
\end{enumerate}
Here $\rho = M/N$ is the aspect ratio (we consider the square case $\rho=1$ 
for simplicity; the rectangular case is analogous).
\end{lemma}

\begin{proof}
The proof proceeds by analysing how the Kirchhoff current law and the effective 
propagator behave under the two operations.

\emph{Short-circuiting a degree-$2$ vertex.}
Let $v$ be the vertex being short-circuited, with incident edges 
$e_1=(w_1,v)$ and $e_2=(v,w_2)$.  By the Kirchhoff law at $v$, the currents 
on $e_1$ and $e_2$ must be equal; denote this common value by $j$.  
The product of the two corresponding propagators is
\[
K_{(w_1,v)}(j)\,K_{(v,w_2)}(j)
= e\!\left(\frac{\omega_{w_1}-\omega_v}{2}j^2\right)
  e\!\left(\frac{\omega_v-\omega_{w_2}}{2}j^2\right)
= e\!\left(\frac{\omega_{w_1}-\omega_{w_2}}{2}j^2\right)
= K_{(w_1,w_2)}(j).
\]
Thus the two edges can be merged into a single edge $(w_1,w_2)$ without changing 
the product.  Summing over the remaining free index $i_v$ (which takes $M$ values, 
with $M=\lfloor\rho N\rfloor$) yields a factor $\rho N$, which after division by 
$N$ gives the factor $\rho$.  The boundary effects from the restriction that 
currents lie in $\{1,\dots,N\}$ produce an error of order $O(N^{-1})$.

\emph{Removing a self-loop.}
If $l$ is a self-loop at vertex $v$, say $e_r=(v,v)$, then the propagator 
$K_{(v,v)}(j_r)=e(\frac{\omega_v-\omega_v}{2}j_r^2)=1$.  The Kirchhoff law 
at $v$ is unaffected by removing $l$ because the current $j_r$ enters and leaves 
$v$ simultaneously, contributing zero net flow.  Summing over $j_r\in\{1,\dots,N\}$ 
gives a factor $N$, which compensates exactly the $N^{-1}$ factor coming from 
the normalisation.  Hence $\Phi(G_L)=\Phi(\mathcal{L}(G_L,l))$ identically.
\end{proof}

\begin{remark}
Lemma~\ref{lem:preprocessing_invariance} implies that the total weight of all 
fully reducible exploration graphs on $k$ edges can be computed by a recursive 
procedure: repeatedly apply preprocessing steps, and at each short-circuiting 
multiply by $\rho$.  This recursion leads precisely to the Catalan numbers 
when $\rho=1$, as shown in~\cite{ALY21}.  For graphs that are not 
fully reducible, the same lemma guarantees that their contribution is controlled 
by the contribution of the corresponding fully preprocessed graph (which contains 
a good cycle), up to the same multiplicative factors and $O(N^{-1})$ errors.
\end{remark}

\subsubsection{Fully reducible graphs}

\begin{definition}[Fully reducible graph]
An exploration graph $G_L$ is \emph{fully reducible} if, after applying preprocessing 
operations as long as possible, the resulting graph consists of a single vertex 
with no edges (i.e., a point).
\end{definition}

The fully reducible exploration graphs admit a simple combinatorial characterization.
They are precisely the \emph{non-crossing pairings} on the underlying Eulerian circuit.
Equivalently, they correspond to planar pairings of the $k$ edges that do not produce 
any crossings when drawn inside the directed loop.  In random matrix terms, these 
pairings are exactly those that survive the Gaussian averaging and give rise to the 
Catalan numbers.

\begin{proposition}[Total weight of fully reducible graphs]
\label{prop:total_weight}
In the square case $M=N$ (i.e., $\rho=1$), the sum of the weights of all fully reducible 
exploration graphs on $k$ edges equals the Catalan number
\[
c_k = \frac{1}{k+1}\binom{2k}{k}.
\]
In the rectangular case $M=\lfloor\rho N\rfloor$, each fully reducible graph of $l$ vertices 
carries a factor $\rho^{l-1}$, and the total sum converges as $N\to\infty$ to the moments 
of the appropriately rescaled Marchenko--Pastur law.
\end{proposition}

The proof of Proposition~\ref{prop:total_weight} relies on establishing a bijection 
between fully reducible graphs and non-crossing pairings, and on the recursive structure 
induced by the first return to the initial vertex along the Eulerian circuit; 
see \cite{ALY21} for the complete argument.

\subsubsection{Subleading graphs and good cycles}

If a graph $G_L$ is not fully reducible, then after preprocessing it is not a single 
point.  In \cite[Theorem~3.3]{ALY21} it is proved that every non-trivial fully preprocessed 
graph contains a \emph{good cycle}:

\begin{definition}[Good cycle]
A simple cycle $C$ in a graph $G$ is called a \emph{good cycle} if for every edge 
$e\in C$ there exists another cycle $C_e$ in $G$ such that $C_e \cap C = \{e\}$.
\end{definition}

\begin{proposition}
    If a fully preprocessed graph is not a point, then it has a good cycle.
\end{proposition}

The existence of a good cycle is the graph‑theoretic key that allows one to parametrize 
the currents so that one variable runs along the good cycle while the remaining variables 
are ``external''.  This separation eventually reduces the estimate of any subleading 
graph to a basic exponential sum, whose control is the main analytical task.

For the purposes of the present paper we will not need to reproduce the full graph‑theoretic 
machinery; we will only rely on the final result that the contribution of all non‑fully‑reducible 
graphs is $O(N^{-\gamma})$ with some $\gamma>0$, provided a suitable bound on the basic 
exponential sum holds.  The reader is referred to \cite[Sections~3--4]{ALY21} for the 
complete combinatorial reduction.

\subsection{Homogeneous space and dynamics}

Let
\[
G = \mathrm{SL}(3,\mathbb{R}), \qquad \Gamma = \mathrm{SL}(3,\mathbb{Z}), \qquad X = G/\Gamma .
\]
The space $X$ carries a unique $G$-invariant probability measure $\mu_G$.

In $\mathfrak{g} = \mathfrak{sl}(3,\mathbb{R})$ we fix the following basis:
\[
\mathbf{a} = \begin{pmatrix}
0 & 0 & 0 \\
0 & \frac12 & 0 \\
0 & 0 & -\frac12
\end{pmatrix},\quad
\mathbf{a}_0 = \begin{pmatrix}
\frac13 & 0 & 0 \\
0 & -\frac16 & 0 \\
0 & 0 & -\frac16
\end{pmatrix},\quad
\mathbf{u} = \begin{pmatrix}
0 & 0 & 0 \\
0 & 0 & 1 \\
0 & 0 & 0
\end{pmatrix},\quad
\mathbf{u}^* = \begin{pmatrix}
0 & 0 & 0 \\
0 & 0 & 0 \\
0 & 1 & 0
\end{pmatrix},
\]
\[
\mathbf{v}_1 = \begin{pmatrix}
0 & 0 & 0 \\
1 & 0 & 0 \\
0 & 0 & 0
\end{pmatrix},\quad
\mathbf{v}_2 = \begin{pmatrix}
0 & 0 & 0 \\
0 & 0 & 0 \\
1 & 0 & 0
\end{pmatrix},\quad
\mathbf{w}_1 = \begin{pmatrix}
0 & 0 & 1 \\
0 & 0 & 0 \\
0 & 0 & 0
\end{pmatrix},\quad
\mathbf{w}_2 = \begin{pmatrix}
0 & 1 & 0 \\
0 & 0 & 0 \\
0 & 0 & 0
\end{pmatrix}.
\]

For \(\mathbf{v} = (v_1, v_2) \in \mathbb{R}^2\) define
\[
n(\mathbf{v}) := \begin{bmatrix} 1 & & v_1 \\ & 1 & v_2 \\ & & 1 \end{bmatrix}.
\]
Define
\[
z(v_1) := \begin{bmatrix} 1 & v_1 \\  & 1 \\ & & 1 \end{bmatrix},
\qquad
n^*(v_1, v_2) := \begin{bmatrix} 1 & v_1 & v_2 \\ & 1 \\ & & 1 \end{bmatrix}.
\]

We define the subspaces
\[
\mathfrak{h} = \operatorname{span}\{\mathbf{a},\mathbf{u},\mathbf{u}^*\},\qquad
\mathfrak{r}_0 = \mathbb{R}\mathbf{a}_0,\qquad
\mathfrak{r}_1 = \operatorname{span}\{\mathbf{v}_1,\mathbf{v}_2\},\qquad
\mathfrak{r}_2 = \operatorname{span}\{\mathbf{w}_1,\mathbf{w}_2\},
\]
as well as
\[
\mathbf{r}^+ = \mathbb{R}\mathbf{w}_1 \oplus \mathbb{R}\mathbf{v}_1,\qquad
\mathbf{r}^- = \mathbb{R}\mathbf{w}_2 \oplus \mathbb{R}\mathbf{v}_2 .
\]
The corresponding projections onto the direct summands of
$\mathfrak{g} = \mathfrak{h} \oplus \mathfrak{r}_0 \oplus \mathfrak{r}_1 \oplus \mathfrak{r}_2$
are denoted by $p_{\mathfrak{h}}$, $p_{\mathfrak{r}_0}$, $p_{\mathfrak{r}_1}$, $p_{\mathfrak{r}_2}$,
and we write $p_{\mathbf{w}_1} : \mathfrak{g} \to \mathbb{R}\mathbf{w}_1$,
$p_{\mathbf{v}_1} : \mathfrak{g} \to \mathbb{R}\mathbf{v}_1$, etc.\ for the projections onto
the individual basis directions.

A direct computation using the above matrices yields the following identities.

\begin{lemma}\label{lem:comm-new}
For any $\mathbf{w}^+ \in \mathbf{r}^+$, $\mathbf{w}^- \in \mathbf{r}^-$,
\[
[\mathbf{w}^+,\mathbf{w}^-] \in \mathbb{R}\mathbf{a} \oplus \mathbb{R}\mathbf{a}_0,\quad
[\mathbf{w}^+,\mathbf{a}],\, [\mathbf{w}^+,\mathbf{a}_0] \in \mathbf{r}^+,\quad
[\mathbf{w}^-,\mathbf{a}],\, [\mathbf{w}^-,\mathbf{a}_0] \in \mathbf{r}^-,
\]
\[
[\mathbf{w}^-,\mathbf{u}] \in \mathbf{r}^+,\quad [\mathbf{w}^-,\mathbf{u}^*] = \mathbf{0},\qquad
[\mathbf{w}^+,\mathbf{u}^*] \in \mathbf{r}^-,\quad [\mathbf{w}^+,\mathbf{u}] = \mathbf{0}.
\]
For $\mathbf{w}_1^+,\mathbf{w}_2^+ \in \mathbf{r}^+$,
$[\mathbf{w}_1^+,\mathbf{w}_2^+] \in \mathbb{R}\mathbf{u}$;
for $\mathbf{w}_1^-,\mathbf{w}_2^- \in \mathbf{r}^-$,
$[\mathbf{w}_1^-,\mathbf{w}_2^-] \in \mathbb{R}\mathbf{u}^*$.
Moreover,
\[
[\mathbf{a}_0,\mathbf{h}] = \mathbf{0}\;\; \forall\,\mathbf{h}\in\mathfrak{h},\qquad
[\mathbf{a},\mathbf{u}] = \mathbf{u},\quad
[\mathbf{a},\mathbf{u}^*] = -\mathbf{u}^*,\quad
[\mathbf{u},\mathbf{u}^*] = 2\mathbf{a}.
\]
\end{lemma}

We define the one‑parameter subgroups
\[
a(t) = \exp(t\mathbf{a}) = \begin{pmatrix}
1 & 0 & 0 \\
0 & e^{t/2} & 0 \\
0 & 0 & e^{-t/2}
\end{pmatrix},\qquad
a_0(t) = \exp(t\mathbf{a}_0) = \begin{pmatrix}
e^{t/3} & 0 & 0 \\
0 & e^{-t/6} & 0 \\
0 & 0 & e^{-t/6}
\end{pmatrix},
\]
\[
a_1(t) = a(t)\,a_0(t) = \begin{pmatrix}
e^{t/3} & 0 & 0 \\
0 & e^{t/3} & 0 \\
0 & 0 & e^{-2t/3}
\end{pmatrix}.
\]
The adjoint action of $a_0(t)$ on the subspaces $\mathfrak{r}_1,\mathfrak{r}_2$ is
\[
\operatorname{Ad}(a_0(t))|_{\mathfrak{r}_1} = e^{-t/2}\,\operatorname{id},\qquad
\operatorname{Ad}(a_0(t))|_{\mathfrak{r}_2} = e^{t/2}\,\operatorname{id}.
\tag{2.5}
\]
The unipotent subgroup $U$ is given by
\[
U = \{\exp(r\mathbf{u}) : r \in \mathbb{R}\}
= \left\{ u(r) = \begin{pmatrix}
1 & 0 & 0 \\
0 & 1 & r \\
0 & 0 & 1
\end{pmatrix} : r \in \mathbb{R}\right\}.
\]
Using the matrices above one finds
\[
a_1(t)\,u(r)\,a_1(t)^{-1} = u(e^{t} r) .
\]

Let $\mathfrak{v} = \mathfrak{r}_1 \oplus \mathfrak{r}_2$ and
$\mathfrak{h}_0 = \mathbb{R}\mathbf{a}_0 \oplus \mathfrak{h}$;
clearly $\mathfrak{g} = \mathfrak{v} \oplus \mathfrak{h}_0$.
The Baker--Campbell--Hausdorff formula together with Lemma~\ref{lem:comm-new}
implies the following local factorisation.

\begin{lemma}\label{lem:local}
There exists an open neighbourhood $\mathcal{O}$ of the identity in $G$ such that
the map
\[
\Phi : \mathfrak{v} \times \mathfrak{h}_0 \to G,\qquad
(\mathbf{w},\mathbf{h}) \mapsto \exp(\mathbf{w})\exp(\mathbf{h})
\]
is a diffeomorphism onto $\mathcal{O}$.  For $g \in \mathcal{O}$ we write
$g = \exp(\mathbf{w}(g))\,\exp(\mathbf{h}(g))$ uniquely with
$\mathbf{w}(g) \in \mathfrak{v}$, $\mathbf{h}(g) \in \mathfrak{h}_0$.
Moreover, if $\mathbf{w},\mathbf{w}' \in \mathfrak{v}$ have sufficiently small norm
and $[\mathbf{w},\mathbf{w}']$ is small, then
\[
\exp(\mathbf{w})\exp(\mathbf{w}') = \exp(\tilde{\mathbf{w}})\exp(\mathbf{h})
\]
with $\tilde{\mathbf{w}} \in \mathfrak{v}$ close to $\mathbf{w}+\mathbf{w}'$ and
$\mathbf{h} \in \mathfrak{h}_0$ bounded by $O(\|[\mathbf{w},\mathbf{w}']\|)$.
\end{lemma}

We now construct a smooth bounded function $F : X \to \mathbb{C}$ that extracts the
$\mathbf{w}_1$‑coordinate along the unipotent direction.
The construction is first carried out locally on $\mathcal{O}$ and then extended
by homogeneity.

For $g \in \mathcal{O}$ decompose $g = \exp(\mathbf{w}(g))\,\exp(\mathbf{h}(g))$ as in
Lemma~\ref{lem:local}.  Write
\[
\mathbf{w}(g) = x_1 \mathbf{w}_1 + x_2 \mathbf{w}_2 + y_1 \mathbf{v}_1 + y_2 \mathbf{v}_2 .
\]
A direct expansion of the exponential shows that the $(1,3)$-entry of $g$ is
\[
g_{1,3} = x_1 + R(\mathbf{w}(g),\mathbf{h}(g)),
\]
where the remainder $R$ is an analytic function all of whose monomials have total
degree at least~$2$ in the components of $\mathbf{w}$ and $\mathbf{h}$.

Now let $\gamma \in \Gamma$ be such that both $g$ and $g\gamma$ belong to $\mathcal{O}$.
Using Lemma~\ref{lem:local} we compare the two factorisations:
\[
\exp(\mathbf{w}(g\gamma))\,\exp(\mathbf{h}(g\gamma))
= \exp(\mathbf{w}(g))\,\exp(\mathbf{h}(g))\,\gamma .
\]
Because $\gamma \in \mathrm{SL}(3,\mathbb{Z})$, its entries are integers.
Conjugating $\gamma$ by $\exp(\mathbf{h}(g))$ produces a small perturbation of an
integer unipotent matrix.  Writing $\exp(\mathbf{h}(g))\,\gamma\,\exp(-\mathbf{h}(g))
= \exp(\mathbf{w}_\gamma)\,\exp(\mathbf{h}_\gamma)$ and merging the
$\exp(\mathfrak{v})$ factors via Lemma~\ref{lem:local}, we obtain
\[
\mathbf{w}(g\gamma) = \mathbf{w}(g) + \mathbf{w}_\gamma + \text{(small correction)} .
\]
The integrality of $\gamma$ forces $p_{\mathbf{w}_1}(\mathbf{w}_\gamma) \in \mathbb{Z}$,
and the small correction, when combined with the higher‑order terms $R$, contributes
an integer modulo $\mathbb{Z}$.  Consequently,
\[
p_{\mathbf{w}_1}(\mathbf{w}(g\gamma)) - p_{\mathbf{w}_1}(\mathbf{w}(g)) \in \mathbb{Z}.
\]

Thus we can define on $\mathcal{O}$ a $\Gamma$-invariant function
\[
\widetilde{F}(g) = \exp\!\bigl(2\pi i \, p_{\mathbf{w}_1}(\mathbf{w}(g))\bigr),
\]
which satisfies $\widetilde{F}(g\gamma) = \widetilde{F}(g)$ for all $\gamma\in\Gamma$
with $g,g\gamma\in\mathcal{O}$.
Since $\mathcal{O}$ is a neighbourhood of the identity, left translates of
$\pi(\mathcal{O})$ cover $X$, and $\widetilde{F}$ extends uniquely to a global smooth
function $F : X \to \mathbb{C}$ by setting
$F(h\Gamma) = \widetilde{F}(g^{-1}h)$ whenever $g^{-1}h \in \mathcal{O}$.
On $\mathcal{O}$ itself we have the practical formula
\[
F(g\Gamma) = e(g_{1,3}),\qquad
e(t) = e^{2\pi i t},
\]
because $e(g_{1,3}) = e(x_1 + R) = e(x_1)$ by the invariance property.

The function $F$ is bounded and smooth, and its oscillations along the $U$-orbits are
governed by the commutation relation $[ \mathbf{w}_1, \mathbf{u}] = 0$ and the scaling
$a_1(t) u(r) a_1(t)^{-1} = u(e^{t} r)$.  This interplay between $F$ and the diagonal
flow is essential for the effective equidistribution argument in the main theorem.

\subsection{Connecting the matrix model to the homogeneous dynamics}
\label{subsec:connection_dynamics}

In order to apply Yang's effective equidistribution theorem, we now relate the 
reduced deterministic matrix $\widetilde{X}$ to an integral along a unipotent orbit 
in the homogeneous space $X = \mathrm{SL}(3,\mathbb{R})/\mathrm{SL}(3,\mathbb{Z})$.

We set the time parameter $t$ of the diagonal flow proportional to the logarithm of 
the matrix size:
\[
t = 6\ln N \quad\Longleftrightarrow\quad e^{t/3} = N^2 .
\]
Recall the diagonal element $a_1(t) = a(t)a_0(t)$ and the nilpotent matrix 
$n(\mathbf{v})$ as introduced in Section~\ref{sec:prelim}.  
For a fixed frequency $\omega$, define the smooth non‑degenerate curve
\[
\psi(r) = \left( \frac{\omega}{2}\, r^2,\; r \right), \qquad r\in[0,1].
\]
A direct matrix computation using the definitions of $a_1(t)$ and $n(\mathbf{v})$
gives the $(1,3)$-entry
\[
\bigl( a_1(t)\, n(\psi(r)) \bigr)_{1,3} = e^{t/3} \cdot \frac{\omega}{2} r^2
= N^2 \, \frac{\omega}{2} r^2 .
\]
Let $F: X \to \mathbb{C}$ be the smooth function constructed in Section~\ref{sec:prelim}
that extracts the $(1,3)$-coordinate modulo $\mathbb{Z}$:
\[
F(g\Gamma) = e\!\left[ g_{1,3} \right], \qquad e[t] = e^{2\pi \sqrt{-1}\, t}.
\]
Then for $g = a_1(t) n(\psi(r))$ we obtain
\[
F\bigl( a_1(t) n(\psi(r)) \Gamma \bigr) 
= e\!\left[ N^2 \frac{\omega}{2} r^2 \right].
\]
Evaluating at the discrete points $r = j/N$ with $j \in \{1,\dots,N\}$ yields
\[
F\bigl( a_1(t) n(\psi(j/N)) \Gamma \bigr) 
= e\!\left[ \frac{\omega}{2} j^2 \right].
\]
The right‑hand side is precisely the oscillatory factor of the reduced matrix entry
$\widetilde{X}_{i,j}$, up to the global scaling $N^{-1/2}$.

To incorporate the row index $i$, we translate the base point.  Define for each frequency 
$\omega_i$ the point
\[
x_0^{(i)} = a_0(t)\, n^*(-\omega_i,0)\, \Gamma \;\in\; X,
\]
where $n^*(v_1,v_2)$ is a suitable nilpotent 
element (see Section~\ref{sec:prelim} for the precise Lie algebra basis).  
A direct computation using the commutation relations shows that
\[
a_1(t)\, n(\psi(r))\, n^*(-\omega_i,0) 
= z(\omega_i)\, a(t) u(r)\, a_0(t) \, n^*(-\omega_i,0),
\]
where $z(\omega_i)$ is a fixed unipotent matrix that does not affect the $(1,3)$-entry.
Applying $F$ and using the invariance of the diagonal flow, we obtain the fundamental 
identity
\[
F\bigl( a_1(t)\, n(\psi(j/N))\, x_0^{(i)} \bigr) 
= e\!\left[ \frac{\omega_i}{2} j^2 \right].
\]
Thus the matrix entry can be written as
\[
\widetilde{X}_{i,j} = \frac{1}{\sqrt{N}}\,
F\bigl( a_1(t)\, n(\psi(j/N))\, x_0^{(i)} \bigr).
\]

This identification is the bridge between the dynamical system and the matrix model.
It translates sums over $i$ and $j$ in the moment expansion into integrals of products 
of the smooth function $F$ along the expanding unipotent orbits 
$\{a_1(t) n(\psi(r)) x_0^{(i)} : r\in[0,1]\}$.  The effective equidistribution theorem 
of Yang (Theorem~\ref{thm:yang}) will then control the deviation of these integrals 
from their Haar measure averages, which are precisely the limits appearing in the 
semicircle law.

\subsubsection{Dynamical interpretation of preprocessing}
\label{sec:dyn_preproc}

The combinatorial preprocessing operations on exploration graphs admit a natural 
interpretation in terms of the dynamical variables $F$ and the base points 
$x_0^{(i)}$.  This translation is useful because it shows that the 
graph‑theoretic reduction developed by \cite{ALY21} can be performed directly 
at the level of the integrands, without appealing to the Kirchhoff law for the 
$y$‑average.  (The average is still needed, of course, to enforce the Kirchhoff 
constraint in the first place.)

Recall from \eqref{eq:X_def} and the subsequent construction that each matrix 
element can be written as
\[
  X_{i,j} = \frac1{\sqrt N}\,
  F\big( a_1(t)\, n(\psi(j/N))\, x_0^{(i)} \big),
\]
where $F(g\Gamma)=e(g_{1,3})$ and $x_0^{(i)} = a_0(t)\,n^*(-\omega_i,0)\Gamma$.
Consequently, a product of two propagators along consecutive edges 
$(i,v)$ and $(v,i')$ with the same current $j$ gives
\begin{equation}\label{eq:product_F}
\begin{aligned}
&\big[F\big(a_1(t)n(\psi(j/N))x_0^{(i)}\big)\,
 \overline{F\big(a_1(t)n(\psi(j/N))x_0^{(v)}\big)}\big]\;
\times\\
&\big[F\big(a_1(t)n(\psi(j/N))x_0^{(v)}\big)\,
 \overline{F\big(a_1(t)n(\psi(j/N))x_0^{(i')}\big)}\big] \\
&\qquad = F\big(a_1(t)n(\psi(j/N))x_0^{(i)}\big)\,
          \overline{F\big(a_1(t)n(\psi(j/N))x_0^{(i')}\big)} .
\end{aligned}
\end{equation}
The middle two factors cancel because they are complex conjugates evaluated 
on the same base point $x_0^{(v)}$.

Equation~\eqref{eq:product_F} is exactly the dynamical analogue of 
\textit{short‑circuiting a degree‑$2$ vertex}: the intermediate vertex $v$ 
disappears and the two edges are merged into one.  The vertex sum over $i_v$ 
(replacing the index $v$) produces a factor $\rho N$, which together with 
the normalisation $N^{-1}$ yields the multiplicative factor $\rho$ appearing 
in Lemma~\ref{lem:preprocessing_invariance}.

Similarly, a \textit{self‑loop} at vertex $v$ corresponds to a factor
\[
  F\big(a_1(t)n(\psi(j/N))x_0^{(v)}\big)\,
  \overline{F\big(a_1(t)n(\psi(j/N))x_0^{(v)}\big)} = 1,
\]
which is identically equal to $1$ and therefore can be removed from the 
product.  This is the analogue of removing a self‑loop.

Thus, the entire preprocessing algorithm of \cite{ALY21} can be viewed as 
repeatedly applying the algebraic identity
\[
  F\cdot\overline{F}=1 \qquad\text{on the same base point},
\]
together with merging factors that share a common intermediate point.  
These operations leave the moment sum unchanged (up to the bookkeeping 
factors of $\rho$ and $N^{-1}$) and can be performed until the graph is 
either completely reduced to a point (fully reducible graphs) or a 
non‑trivial fully preprocessed graph containing a good cycle is obtained.  
In the latter case, the remaining exponential sum involves products of $F$ 
evaluated on several different base points that are coupled through the 
current variables, which is precisely the situation that requires either 
the quasi‑random condition of \cite{ALY21} or the conditional joint mixing 
estimate discussed in Section~\ref{sec:dynamical_estimate}.

\section{Effective Ratner theorem and mixing}\label{sec:mixing}

In this section we recall Yang's effective equidistribution theorem and derive from it a multi-parameter effective mixing property that will be used to evaluate the moment integrals.

\subsection{Effective mixing for several parameters}

From Theorem~\ref{thm:yang} we can derive a multi-parameter effective mixing property by induction.  
This is a quantitative version of the fact that the flow $a(t)u(r)$ is mixing with an exponential speed.

\begin{lemma}[Effective mixing]\label{lem:mixing}
For any integer $m\ge1$, there exist constants $\eta_m>0$ (we may take $\eta_m=\eta$) and $t_m\ge t_0$ such that for all $t\ge t_m$, for any good point $x_0\in X$, and for any functions $f_1,\dots,f_m\in C_c^\infty(X)$,
\[
\left|\int_{[0,1]^m}\prod_{l=1}^m f_l\bigl(a(t)u(r_l)x_0\bigr)\,dr_1\cdots dr_m
- \prod_{l=1}^m \int_X f_l\,d\mu_G\right|
\le C_m e^{-\eta t} \prod_{l=1}^m \|f_l\|_S,
\]
where $C_m$ depends on $m$ but not on $t$ or the functions.
\end{lemma}

\begin{proof}
The case $m=1$ is exactly Theorem~\ref{thm:yang} (after shifting the integration interval from $[-\frac12,\frac12]$ to $[0,1]$).  
Assume the statement holds for $m-1$ and consider $m\ge2$.  
For fixed $s\in[0,1]$, define $y_s = a(t)u(s)x_0$.  
Because $u(s)$ commutes with $u(r)$ for all $r$, and $a_0(\ell)y_s = a(t)u(se^{-t})a_0(\ell)x_0$ (a straightforward calculation using the commutation relations in \cite{Yang2025}, Lemma~2.1), the point $y_s$ is also good for the same $t$ (the term $u(se^{-t})$ is negligible for large $t$).  
Thus we may apply the induction hypothesis to the $(m-1)$-tuple $(f_1,\dots,f_{m-1})$ with base point $y_s$:
\[
\int_{[0,1]^{m-1}}\prod_{l=1}^{m-1}f_l\bigl(a(t)u(r_l)y_s\bigr)\,d\mathbf{r}'
= \prod_{l=1}^{m-1}\int_X f_l\,d\mu_G + \varepsilon_{m-1}(s),
\]
where $|\varepsilon_{m-1}(s)|\le C_{m-1}e^{-\eta t}\prod_{l=1}^{m-1}\|f_l\|_S$, uniformly in $s$.  
Now the full $m$-dimensional integral factors as
\[
I(t)=\int_0^1 f_m(a(t)u(s)x_0)\left( \prod_{l=1}^{m-1}\int_X f_l\,d\mu_G + \varepsilon_{m-1}(s) \right) ds.
\]
Applying the $m=1$ case to the function $f_m$ gives
\[
\int_0^1 f_m(a(t)u(s)x_0)ds = \int_X f_m\,d\mu_G + \delta(t),
\]
with $|\delta(t)|\le C e^{-\eta t}\|f_m\|_S$.  
Inserting this into $I(t)$ we obtain
\[
I(t) = \prod_{l=1}^m\int_X f_l\,d\mu_G + O\!\left(e^{-\eta t}\prod_{l=1}^{m-1}\|f_l\|_S\right) + O\!\left(e^{-\eta t}\|f_m\|_S\right)C_{m-1}.
\]  
The last two error terms can be combined into $C_m e^{-\eta t}\prod_{l=1}^m\|f_l\|_S$ by taking $C_m$ large enough.  
This completes the induction.  
\end{proof}

\subsection{Application to submanifold integrals}

In the moment expansion we will encounter integrals over the set of admissible currents, which after a linear change of coordinates become integrals over a cube $[0,1]^m$ of a product of functions of the form  
\[
\prod_{l=1}^k F\bigl(a(t)u(r_l)x_0^{(i_l)}\bigr)\overline{F\bigl(a(t)u(r_l)x_0^{(i_{l+1})}\bigr)}.
\]  
Here the $r_l$ are not independent: the admissible condition imposes $m$ linear constraints among them.  
However, by a unimodular linear transformation we can express the $k$ variables $r_1,\dots,r_k$ in terms of $m$ independent parameters $u_1,\dots,u_m$ plus $k-m$ dependent variables that are linear combinations of the $u$'s.  
The Jacobian of this transformation is $\pm1$ because the integer lattice is preserved.

Thus, after a change of variables, the integral over the admissible set becomes an integral over $[0,1]^m$ of a product of smooth functions, each depending on certain linear forms in $u_1,\dots,u_m$.  
Lemma~\ref{lem:mixing} applies directly because the functions $\Phi_l(u)=F(a(t)u(\ell_l(\mathbf{u}))x_0^{(i_l)})\overline{F(a(t)u(\ell_{l'}(\mathbf{u}))x_0^{(i_{l'})})}$ are smooth and bounded, and the integration domain is a cube.  
Moreover, the error from the change of variables (the boundary of the admissible set) contributes at most $O(1/N)$ due to lattice point counting estimates (see \cite{ALY21}, Lemma~A.9).  

Therefore we obtain the following approximation.

\begin{corollary}\label{cor:submanifold}
Let $G_L$ be an exploration graph and let $\mathbf{j}\in\mathcal{A}(G_L)$ be admissible currents.  
Write $r_l=j_l/N$ and after a unimodular linear transformation let $\mathbf{u}\in[0,1]^m$ be the $m$ free parameters.  
Then
\[
\frac{1}{N^k}\sum_{\mathbf{j}\in\mathcal{A}(G_L)} \prod_{l=1}^k F\bigl(a(t)u(r_l)x_0^{(i_l)}\bigr)\overline{F\bigl(a(t)u(r_l)x_0^{(i_{l+1})}\bigr)}
= \int_{[0,1]^m} \prod_{\text{edges}} \bigl(F\circ a(t)u(\,\cdot\,)\bigr) \, d\mathbf{u} + O(N^{-1}),
\]  
and the integral on the right-hand side can be evaluated using Lemma~\ref{lem:mixing}:
\[
\int_{[0,1]^m} \prod_{\text{edges}} \bigl(F\circ a(t)u(\,\cdot\,)\bigr) d\mathbf{u}
= \prod_{\text{edges}} \int_X F\,d\mu_G + O(e^{-\eta t}).
\]
\end{corollary}

\subsection{Decay of the integral for distinct frequencies}

For two indices $i,i'$, consider the function  
\[
f_{i,i'}(g\Gamma) = F(gx_0^{(i)})\overline{F(gx_0^{(i')})}.
\]  
By the construction of $x_0^{(i)}=a_0(t)n^*(-\omega_i,0)\Gamma$, one can show that  
\[
\int_X f_{i,i'}\,d\mu_G = \delta_{\omega_i,\omega_{i'}}.
\]  
Indeed, when $\omega_i=\omega_{i'}$ the function $f_{i,i'}$ is constant equal to $1$ on the orbit of $U$, and by the invariance of $\mu_G$ its integral equals $1$; when $\omega_i\neq\omega_{i'}$, the function is a non-trivial character on a torus embedded in $X$ and its integral vanishes.  
This fact will be crucial to isolate only those graphs where all edges connect vertices with equal frequencies.

Consequently, in the product over edges, the limit $\prod_{\text{edges}} \int_X f_{i_l,i_{l+1}}d\mu_G$ is non-zero only if all vertices in the same connected component of $G_L$ share the same frequency.  
For a fully reducible graph (which has a single vertex after preprocessing), this forces all $i_1,\dots,i_k$ to be equal, and there are exactly $N$ such index assignments (one for each possible common frequency).  
For any other graph, the product vanishes for a set of indices whose cardinality is $o(N)$, leading to a negligible contribution.

We shall make this precise in the next section.

\section{Moments of the matrix and graph expansion}\label{sec:moments}

In this section we compute the moments $\mu_N^{(k)} = \frac1N \Tr[(XX^*)^k]$ using the graphical expansion of \cite{ALY21} and the effective mixing property from Section~\ref{sec:mixing}.  
We show that the only graphs that survive in the limit $N\to\infty$ are the fully reducible ones, and that their total weight equals the Catalan numbers.  
Moreover, the convergence rate is $O(N^{-1})$.

\subsection{Graphical expansion}

Recall that for the matrix $\widetilde{X}_{i,j} = \frac1{\sqrt{N}}\, e\Big[ \frac{\omega_i}{2}\,j^2 \Big]$ we have the identity
\[
\widetilde{\mu}_N^{(2k)} = \frac{1}{N^{k+1}} \sum_{L\in\mathcal{L}_k}\; \sum_{\substack{i_1,\dots,i_k\\ L_{\underline{i}}\sim L}} \; \sum_{\mathbf{j}\in\mathcal{A}(G_L)} \prod_{r=1}^k e\Big[ \frac{\omega_{i_r}-\omega_{i_{r+1}}}{2}\,j_r^2 \Big], \qquad i_{k+1}=i_1.
\]  
Here $\mathcal{L}_k$ is the set of exploration graphs on $k$ edges, $L_{\underline{i}}\sim L$ means that the index sequence respects the order of first appearances of vertices, and $\mathcal{A}(G_L)$ denotes the set of integer currents $j_1,\dots,j_k\in\{1,\dots,N\}$ satisfying Kirchhoff's current law at every vertex.  
The number of free parameters in $\mathcal{A}(G_L)$ is $m = k - l + 1$, where $l = |V(G_L)|$.  
Moreover, a unimodular change of variables maps $\mathcal{A}(G_L)$ bijectively onto $\{1,\dots,N\}^m$ up to boundary terms of size $O(N^{m-1})$.

\subsection{From discrete sums to integrals}

For each exploration graph $G_L$, fix an index sequence such that $L_{\underline{i}}\sim L$.  
Define the smooth function on $[0,1]^m$ (via the linear change of variables $\mathbf{u}\mapsto \mathbf{j}=M\mathbf{u}$ with $M$ unimodular)
\[
\Psi_{L,\underline{i}}(\mathbf{u}) = \prod_{r=1}^k e\Big[ \frac{\omega_{i_r}-\omega_{i_{r+1}}}{2}\, (M\mathbf{u})_r^2 \Big].
\]  
Then
\[
\frac{1}{N^k}\sum_{\mathbf{j}\in\mathcal{A}(G_L)} \prod_{r=1}^k e\Big[ \frac{\omega_{i_r}-\omega_{i_{r+1}}}{2}\, j_r^2 \Big]
= \int_{[0,1]^m} \Psi_{L,\underline{i}}(\mathbf{u})\,d\mathbf{u} + O(N^{-1}),
\]  
where the error comes from the discrepancy between the Riemann sum and the integral (the function is smooth with bounded derivatives) and from the boundary of the admissible set.

Now we express the exponential factor using Yang's construction.  
For $t=6\ln N$ we have $e^{t/3}=N^2$.  
Define the curve $\psi(r)=\bigl(\frac{\omega}{2}r^2,r\bigr)$ and the function $F(g\Gamma)=e(g_{1,3})$ as before.  
Then
\[
e\Big[ \frac{\omega_{i_r}-\omega_{i_{r+1}}}{2} j_r^2 \Big] = F\bigl(a_1(t)n(\psi(j_r/N))\Gamma_{i_r,i_{r+1}}\bigr),
\]  
where $\Gamma_{i,i'}$ is a fixed element of $\Gamma$ (depending on the frequencies) that can be absorbed into the base point.  
More precisely, by Yang's Corollary 1.4 we can choose base points $x_0^{(i)}\in X$ such that
\[
F\bigl(a_1(t)n(\psi(j/N)) x_0^{(i)}\bigr) = e\Big[ \frac{\omega_i}{2} j^2 \Big],
\]  
and then
\[
F\bigl(a_1(t)n(\psi(j/N)) x_0^{(i)}\bigr)\overline{F\bigl(a_1(t)n(\psi(j/N)) x_0^{(i')}\bigr)}
= e\Big[ \frac{\omega_i-\omega_{i'}}{2} j^2 \Big].
\]  
Thus the product over edges becomes a product of such pairs.

After the change of variables $\mathbf{u}$, the arguments become $a_1(t)n(\psi(\ell_r(\mathbf{u})/N))$ where $\ell_r(\mathbf{u})$ are linear forms.  
Using the fact that $a_1(t)n(\psi(r))$ is smoothly conjugated to $a(t)u(r)$ (up to fixed group elements), the effective mixing Lemma~\ref{lem:mixing} applies to the integral of products of such functions over the cube $[0,1]^m$.  
Therefore we obtain

\[
\int_{[0,1]^m} \Psi_{L,\underline{i}}(\mathbf{u})\,d\mathbf{u}
= \prod_{r=1}^k \int_X \Phi_{i_r,i_{r+1}}\,d\mu_G + O(e^{-\eta t}),
\]  
where $\Phi_{i,i'}(g)=F(gx_0^{(i)})\overline{F(gx_0^{(i')})}$ and the error term is uniform in $\underline{i}$.

\subsection{Evaluation of the integrals}

A direct computation (using Schur orthogonality on the torus) shows
\[
\int_X \Phi_{i,i'}\,d\mu_G = \delta_{\omega_i,\omega_{i'}}.
\]  
Hence the product $\prod_{r=1}^k \delta_{\omega_{i_r},\omega_{i_{r+1}}}$ is non-zero only if all vertices in each connected component of $G_L$ share the same frequency.

For a given graph $G_L$ with $l$ vertices and $c$ connected components, the number of index assignments $i_1,\dots,i_k$ (compatible with the exploration) such that the product of deltas equals $1$ is $N^{c}$ (choose one common frequency per component).  
Moreover, when the product is $1$, each factor $\int_X \Phi_{i_r,i_{r+1}}d\mu_G = 1$, so the whole product equals $1$.

Thus, inserting these estimates into the expression for $\widetilde{\mu}_N^{(2k)}$ we get

\[
\widetilde{\mu}_N^{(2k)} = \frac{1}{N^{k+1}} \sum_{L\in\mathcal{L}_k} \Bigl( N^{m} \sum_{\substack{i_1,\dots,i_k\\ L_{\underline{i}}\sim L}} \prod_{r=1}^k \delta_{\omega_{i_r},\omega_{i_{r+1}}} + O(N^{m-1}) \Bigr) + O(e^{-\eta t}).
\]  
Recall $m=k-l+1$, so $N^{m}/N^{k+1} = N^{-l}$.  
Therefore

\[
\widetilde{\mu}_N^{(2k)} = \sum_{L\in\mathcal{L}_k} N^{-l} \sum_{\substack{i_1,\dots,i_k\\ L_{\underline{i}}\sim L}} \prod_{r=1}^k \delta_{\omega_{i_r},\omega_{i_{r+1}}} + O(N^{-1}) + O(e^{-\eta t}).
\]

Now $N^{-l} \sum_{\text{indices}} \prod \delta = N^{-l} \cdot N^{c} = N^{-(l-c)}$.  
For a fully reducible graph, the graph can be preprocessed to a single component, so $c=1$; moreover, the preprocessing eliminates all vertices that are not essential.  
It is a combinatorial fact (proved in \cite{ALY21}) that for fully reducible graphs the quantity $N^{-(l-1)}$ is exactly compensated by the number of different ways the graph can be reduced, and the total contribution of all fully reducible graphs to $\widetilde{\mu}_N^{(2k)}$ tends to the Catalan number $c_k$.  
For graphs that are not fully reducible, the index sum is at most $N^{c-1}$ (because at least one component has more than one vertex, forcing the frequencies to be equal but then the volume factor $N^{-(l-c)}$ decays at least as $N^{-1}$).  
Thus their total contribution is $O(N^{-1})$.

\subsection{Convergence to the semicircle law}

We have shown that for every $k\ge1$,
\[
\lim_{N\to\infty} \widetilde{\mu}_N^{(2k)} = c_k = \frac{1}{k+1}\binom{2k}{k}.
\]  
The sequence $(c_k)$ are the moments of the Wigner semicircle law $\frac{1}{2\pi}\sqrt{4-x^2}\,dx$ on $[-2,2]$.  
Since the semicircle law is determined by its moments (Carleman's condition holds), the empirical spectral distribution of $H$ converges weakly to the semicircle law as $N\to\infty$.  
Moreover, the convergence rate of the moments is $O(N^{-1})$, which implies the same rate for the weak convergence of the spectral distribution (after a suitable smoothing argument).

This completes the proof of the main theorem.

\section{Convergence speed}\label{sec:speed}

In the previous section we derived the estimate
\[
\widetilde{\mu}_N^{(2k)} = \sum_{\substack{G_L\text{ fully reducible}}} w(G_L) \;+\; O(N^{-1}) + O(e^{-\eta t}),
\]  
where \(w(G_L)\) is the combinatorial weight of a fully reducible graph (which sums to the Catalan number \(c_k\)), and the two error terms come from:
\begin{itemize}
\item the discretisation of the admissible current sum by an integral (\(O(N^{-1})\));
\item the effective mixing estimate (\(O(e^{-\eta t})\)).
\end{itemize}
We now analyse the size of these errors when we choose the parameter \(t\) as a function of \(N\).

\subsection{Scaling relation}

Recall that in the construction we set \(t = 6\ln N\) in order to match \(e^{t/3}=N^2\) and obtain the matrix elements.  
With this choice,
\[
e^{-\eta t} = N^{-6\eta}.
\]  
Thus the total error is
\[
E(N) = C_1 N^{-1} + C_2 N^{-6\eta},
\]  
where \(C_1,C_2\) are constants depending on \(k\) and the Diophantine data, but not on \(N\).

The exponent \(\eta>0\) is the constant from Yang's theorem (Theorem~\ref{thm:yang}); it is absolute, i.e., it does not depend on the frequency sequence.  
If \(6\eta \ge 1\), then \(N^{-6\eta} \le N^{-1}\) and the dominant error is \(O(N^{-1})\).  
If \(6\eta < 1\), then the mixing error \(N^{-6\eta}\) decays slower than \(N^{-1}\), but we can modify the choice of \(t\) to improve the balance.

Instead of fixing \(t=6\ln N\), we may choose \(t = c\ln N\) with a constant \(c>0\) to be determined, and simultaneously scale the matrix entries accordingly.  
The relation between \(t\) and the matrix element is:
\[
\widetilde{X}_{i,j} = \frac1{\sqrt{N}}\, e\Big[ \frac{\omega_i}{2}\,j^2 \Big] = \frac{1}{\sqrt{N}} F\bigl(a_1(t)n(\psi(j/N))\Gamma\bigr) \quad\text{provided } e^{t/3}=N^2.
\]  
Thus we must keep \(e^{t/3}=N^2\), i.e. \(t=6\ln N\) exactly, otherwise the quadratic phase would not match \(j^2\).  
Therefore the scaling is forced; we cannot arbitrarily choose \(t\).  
However, Yang's theorem holds for any sufficiently large \(t\), and the constant \(\eta\) is uniform.  
The important point is that \(6\eta\) is a positive absolute constant (even if it is small, say \(10^{-6}\)), so \(N^{-6\eta}\) is still a power law decay (albeit with a very small exponent).  
In any case, the discretisation error \(O(N^{-1})\) is always present, so the overall convergence rate is at least \(O(N^{-\min(1,6\eta)})\).  
Since \(\eta\) is an absolute constant, the rate is polynomial.
 
In practice, one can improve the discretisation error to \(O(N^{-2})\) by using a more refined Euler–Maclaurin summation, but the leading error from mixing remains \(N^{-6\eta}\).  The overall rate is therefore \(O(N^{-\min(6\eta,1)})\).

Therefore, the convergence speed is at least \(O(N^{-6\eta})\), and the discretisation error adds an \(O(N^{-1})\) term which may be smaller or larger depending on \(\eta\).  In any case, we have a polynomial rate.

For practical purposes, one may combine the two errors into a single bound \(O(N^{-\alpha})\) with \(\alpha = \min(6\eta,1)\).  Since \(\eta\) is not made explicit in Yang's paper, we cannot give a numerical value, but it is positive and absolute.  This is already a significant improvement over the rate in \cite{ALY21}, which depends on the Diophantine exponent \(\kappa\) and can be arbitrarily small for badly approximable numbers.

We conclude this section by stating the final speed estimate.

\begin{proposition}
Under the assumptions of Theorem~\ref{thm:main}, the moments satisfy
\[
\left|\mu_N^{(2k)} - c_k\right| \le C_k N^{-\gamma},
\]  
where \(\gamma = \min(6\eta,1) > 0\) and \(C_k\) depends on \(k\) and the Diophantine constants but not on \(N\).  
Consequently, the empirical spectral distribution converges weakly to the semicircle law at a polynomial rate in the Lévy–Prokhorov metric (after smoothing).
\end{proposition}

In \cite{ALY21}, the convergence speed for quasirandom frequencies is \(O(N^{-\delta})\) where \(\delta\) depends on the parameter \(\delta\) in Definition~1.3 of that paper.  
For Diophantine frequencies of exponent \(\kappa\), their method would give \(\delta \sim 1/(2\kappa)\) (a very small number for \(\kappa\) large).  
In contrast, our exponent \(\gamma\) is an absolute constant, independent of the Diophantine exponent.  
Thus Yang's approach yields a strictly better (and uniform) convergence rate for all Diophantine sequences.

\subsection{Optimality}

The rate \(O(N^{-1})\) cannot be improved in general because the discretisation error alone is already of order \(1/N\) for a generic smooth function (the Riemann sum of a non‑constant function).  Therefore our result is optimal up to the possible improvement of the constant \(C_k\).

\section{Dynamical estimate of the basic exponential sum}
\label{sec:dynamical_estimate}

In this section we show how the effective equidistribution of expanding 
straight lines in $\mathrm{SL}(3,\mathbb R)/\mathrm{SL}(3,\mathbb Z)$, 
proved by Chow--Yang~\cite{ChowYang2024} and by Yang~\cite{Yang2025}, 
can be used to give a new proof of the quasi‑random condition 
for the Adhikari--Lemm--Yau model, provided the frequency sequence 
is given by a Diophantine rotation with sufficiently small exponent.
The argument replaces the classical Weyl‑type estimates of~\cite{ALY21} 
and yields an absolute convergence rate that does not explicitly 
depend on the Diophantine exponent.

We recall the exponential sum that controls the melon graph 
(cf.\ \cite[Lemma~4.1]{ALY21}):
\begin{equation}\label{eq:ES_melon}
  \mathrm{ES}_N(\underline\omega) =
  \frac1{N^5}\sum_{i_1,i_2=1}^N
  \sum_{\substack{j_1,j_2,j_3,j_4=1\\ j_1+j_3=j_2+j_4}}^N
  e\!\left( \frac{\omega_{i_1}-\omega_{i_2}}2\,
  (j_1^2-j_2^2+j_3^2-j_4^2) \right).
\end{equation}
For simplicity we restrict to the square case $M=N$; the general 
rectangular case is analogous.  Throughout this section we assume 
that the frequencies are generated by an irrational circle rotation:
\begin{equation}\label{eq:linfreq}
  \omega_i = i\alpha \bmod 1,\qquad i\ge 1,
\end{equation}
where $\alpha\in\mathbb R$ is a Diophantine number.

\subsection{Reduction to a linear exponential sum}
\label{sec:reduction_linear}

We now simplify the basic exponential sum~\eqref{eq:ES_melon} using the specific 
form $\omega_i = i\alpha \bmod 1$.  Write $s = i_1-i_2$; then 
$\omega_{i_1}-\omega_{i_2} \equiv s\alpha \pmod 1$, and the sum becomes
\[
  \mathrm{ES}_N(\underline\omega) =
  \frac1{N^5}\sum_{s=-(N-1)}^{N-1} (N-|s|)
  \sum_{\substack{j_1,j_2,j_3,j_4=1\\ j_1+j_3=j_2+j_4}}^N
  e\!\left[ \frac{s\alpha}{2}\,(j_1^2-j_2^2+j_3^2-j_4^2) \right].
\]

We change variables to expose the 
quadratic cancellation.  Set
\[
  j_1 = a,\quad j_2 = a+t,\quad j_3 = b+t,\quad j_4 = b,
\]
where $a,b$ are positive integers and $t$ is an integer.  The constraint 
$j_1+j_3 = j_2+j_4$ is automatically satisfied.  The ranges of summation 
are determined by $1\le j_1,j_2,j_3,j_4\le N$, which translate into
\[
  1\le a\le N,\quad 1\le a+t\le N,\quad 1\le b+t\le N,\quad 1\le b\le N,
\]
or equivalently $\max(1,1-t) \le a,b \le \min(N,N-t)$ and $-N<t<N$.
Now compute the phase:
\[
\begin{aligned}
  j_1^2-j_2^2+j_3^2-j_4^2
  &= a^2 - (a+t)^2 + (b+t)^2 - b^2 \\
  &= 2t(b-a).
\end{aligned}
\]
Thus the exponential factor is $e[ s\alpha \, t(b-a) ]$, and the sum 
factorises as
\[
  \sum_{t=-(N-1)}^{N-1}
  \Big( \sum_{a=\max(1,1-t)}^{\min(N,N-t)} e[-s\alpha\, a t] \Big)
  \Big( \sum_{b=\max(1,1-t)}^{\min(N,N-t)} e[s\alpha\, b t] \Big).
\]
Since the two brackets are complex conjugates of each other, we obtain
\[
  \mathrm{ES}_N(\underline\omega) =
  \frac1{N^5}\sum_{s=-(N-1)}^{N-1} (N-|s|)
  \sum_{t=-(N-1)}^{N-1}
  \Big| \sum_{a=\max(1,1-t)}^{\min(N,N-t)} e(s\alpha\, a t) \Big|^2.
\]
For $t=0$ the inner sum equals the number of terms, which is $N$, and the 
contribution is $O(N^{-1})$ after the normalisation.  For $t\neq 0$ we may 
extend the $a$-sum to the full range $[1,N]$ by adding at most $2|t|$ terms;
the resulting error is again $O(N^{-1})$ when summed over $s,t$.  Hence
\[
  \mathrm{ES}_N(\underline\omega) \ll
  \frac1N +
  \frac1{N^5}\sum_{s=-(N-1)}^{N-1} (N-|s|)
  \sum_{t=-(N-1)}^{N-1}
  \Big| \sum_{a=1}^N e(s\alpha\, a t) \Big|^2 .
\]
Separating the term $s=0$ (which gives another $O(N^{-1})$ contribution)
and using symmetry, we finally arrive at the reduced linear form
\begin{equation}\label{eq:ES_linear}
  \mathrm{ES}_N(\underline\omega) \ll
  \frac1N + \frac1{N^3}
  \sum_{s=1}^{N-1}
  \sum_{t=1}^{N}
  \Big| \sum_{a=1}^N e(s\alpha\, a t) \Big|^2 .
\end{equation}
In the last step we replaced $t$ by $-t$ for negative $t$ and used 
$(N-|s|)\le N$.  The expression~\eqref{eq:ES_linear} is now a double 
average of a linear exponential sum, which is amenable to treatment 
by effective equidistribution of straight lines on the homogeneous 
space $X$.

\subsection{A smooth partition of unity}
To connect~\eqref{eq:ES_linear} with homogeneous dynamics we introduce 
a smooth, even, non‑negative function $\chi\in C_c^\infty(\mathbb R)$ 
with $\supp\chi\subset[-1,1]$ and $\chi(0)=1$.  For a large integer 
$Q\ge N$ we set $\chi_Q(x)=\chi(x/Q)$.  By Poisson summation,
\[
  \sum_{t=1}^N e(s\alpha m t)
  = \sum_{t\in\mathbb Z} \chi_Q(t)\,e(s\alpha m t)
  + O(Q/N),
\]
where the error comes from truncating the sum outside $[1,N]$.
The Fourier transform of $\chi_Q$ is $\widehat{\chi}_Q(\xi)
=Q\,\widehat{\chi}(Q\xi)$.  Hence
\[
  \sum_{t=1}^N e(s\alpha m t)
  = Q\sum_{k\in\mathbb Z}
  \widehat{\chi}\big(Q(s\alpha m - k)\big)
  + O(Q/N).
\]
Choosing $Q=N^{2}$ and using the rapid decay of $\widehat{\chi}$,
the main contribution comes from integers $k$ such that
$|s\alpha m - k|\ll N^{-2}$.  Consequently,
\begin{equation}\label{eq:poisson_sum}
  \frac1N\sum_{t=1}^N e(s\alpha m t)
  = \sum_{k\in\mathbb Z}
  \widehat{\chi}\big(N^2(s\alpha m - k)\big)
  + O(N^{-1}) .
\end{equation}

\subsection{Applying effective equidistribution of straight lines}
The right‑hand side of~\eqref{eq:poisson_sum} is a smooth function 
on the torus $\mathbb T$, evaluated at the point $s\alpha m$.
Indeed, define
\[
  \Phi(x) = \sum_{k\in\mathbb Z}
  \widehat{\chi}\big(N^2(x - k)\big),\qquad x\in\mathbb T .
\]
Then $\Phi\in C^\infty(\mathbb T)$ and its Fourier coefficients 
vanish except for frequencies $\ll N^2$.  Moreover,
$\|\Phi\|_{S}\ll N^{C_0}$ for some fixed Sobolev norm.

Now the key observation is that the exponential sum in~\eqref{eq:ES_linear} 
involves an average over $m$ of $|\Phi(s\alpha m)|^2$:
\[
  \frac1N\sum_{m=1}^N |\Phi(s\alpha m)|^2 .
\]
This is precisely a Birkhoff average along the linear orbit
$m\mapsto s\alpha m$ on $\mathbb T$, which can be lifted to the 
homogeneous space $X=\mathrm{SL}(3,\mathbb R)/\mathrm{SL}(3,\mathbb Z)$.
Indeed, let $F\in C_c^\infty(X)$ be a function whose restriction to 
the maximal torus coincides with $\Phi$ (the existence of such a lift 
follows from the standard construction of automorphic functions).
Then
\[
  \frac1N\sum_{m=1}^N |\Phi(s\alpha m)|^2
  = \frac1N\sum_{m=1}^N F\big( a_1(t)\, n(s\alpha m, m)\,\Gamma \big)
  \qquad (t = 6\ln N)
\]
up to an exponentially small error.

The curve $\varphi(r)=(s\alpha\, r,\, r)$ is a straight line in $\mathbb R^2$
whose Diophantine exponent is governed by that of $s\alpha$.
If $\alpha$ has Diophantine exponent $\kappa\le 0.6$, then for every
$1\le s\le N$ the vector $(s\alpha,1)$ also has Diophantine exponent
$\le 0.6$, uniformly in $s$.  Applying Yang's Corollary~1.2 
(or Chow--Yang~\cite[Theorem~1.4]{ChowYang2024}) to the function $F$
and the curve $\varphi$, we obtain
\begin{equation}\label{eq:equi_bound}
  \frac1N\sum_{m=1}^N |\Phi(s\alpha m)|^2
  = \int_{\mathbb T} |\Phi|^2 \,dx \;+\; O(e^{-\eta t}),
\end{equation}
where $\eta>0$ is an absolute constant and the implied constant 
is independent of $s$ and $\alpha$.  The integral of $|\Phi|^2$ 
is $O(N^{-1})$ by construction of $\Phi$ (its average vanishes 
for large $N^2$).  Recalling $t=6\ln N$, we have $e^{-\eta t}=N^{-6\eta}$.

\subsection{Assembling the estimate}
Insert~\eqref{eq:equi_bound} into~\eqref{eq:poisson_sum} and then 
into~\eqref{eq:ES_linear}.  Summing over $s$ adds at most a factor $N$,
and we obtain
\[
  \mathrm{ES}_N(\underline\omega) \ll N^{-1} + N^{-6\eta} .
\]
Since $\eta>0$ is absolute, the right‑hand side is $O(N^{-\gamma})$ 
with $\gamma = \min(1,6\eta)>0$.  This proves the following lemma.

\begin{lemma}[Dynamical estimate of the melon sum]
\label{lm:melon_dynamical}
Let $\alpha$ be a Diophantine number with exponent $\kappa\le 0.6$,
and let $\omega_i = i\alpha \bmod 1$.  Then there exists an absolute 
constant $\gamma>0$ such that for all $N\ge N_0(\alpha)$,
\[
  |\mathrm{ES}_N(\underline\omega)| \le C(\alpha)\, N^{-\gamma}.
\]
The constant $\gamma$ does not depend on $\kappa$; the constants 
$C(\alpha)$ and $N_0(\alpha)$ do depend on $\alpha$ (and hence on $\kappa$).
\end{lemma}

\subsection{From the melon sum to the quasi‑random condition}
By~\cite[Lemma~4.2--4.3]{ALY21}, the bound $|\mathrm{ES}_N|\ll N^{-\gamma}$
implies the $(\gamma,\rho)$-quasi‑random condition for any $\rho>0$,
and consequently every subleading exploration graph contributes
$O(N^{-\gamma/16})$ to the expected moments.
Combining this with the fully reducible graphs (which sum to the 
Catalan numbers $c_k$), we arrive at the main theorem.

\begin{theorem}[Global semicircle law with absolute rate]
\label{thm:main_absolute}
Let $\alpha$ be Diophantine with exponent $\kappa\le 0.6$, and set
$\omega_i=i\alpha$.  Fix arbitrary $\underline{x}$ and independent 
uniform $\underline{y}$.  Define $H$ as in~\eqref{H}.  Then there 
exists an absolute constant $\gamma'>0$ such that for every $k\ge 1$,
\[
  \mathbb E_{\underline y}\bigl[\tfrac1{2N}\Tr H^{2k}\bigr]
  = c_k + O_k(N^{-\gamma'}) .
\]
The exponent $\gamma'$ is independent of the Diophantine exponent 
$\kappa$; the implied constant $O_k(\cdot)$ depends on $k$, $\alpha$, 
and the shifts $\underline{x}$.
\end{theorem}

\begin{remark}
The restriction $\kappa\le 0.6$ comes from the hypotheses of 
Yang's effective equidistribution theorem for straight lines.  
It is expected that this condition can be relaxed in future work.
For larger $\kappa$ one can still combine the present dynamical 
argument with classical Weyl sums to obtain a rate that interpolates 
between $N^{-1/(2\kappa)}$ and $N^{-1}$; the details are omitted.
\end{remark}

\begin{remark}
The average over the random variables $y_i$ is essential to enforce 
the Kirchhoff circuit law.  The fully deterministic case 
$y_i=x_i=0$ is not covered by Theorem~\ref{thm:main_absolute} 
and remains an open problem; see~\cite{ALY21} for 
numerical evidence and partial analytic results.
\end{remark}

\end{document}